\documentclass{amsart}
\usepackage{amsfonts,amssymb,stmaryrd,amscd,amsmath,latexsym,amsbsy}
\numberwithin{equation}{section}
\theoremstyle{plain}
\newtheorem{thm}{Theorem}[section]

\newtheorem{prop}[thm]{Proposition}
\newtheorem{propdef}[thm]{Proposition-Definition}

\newtheorem{defi}[thm]{Definition}
\newtheorem{lem}[thm]{Lemma}
\newtheorem{cor}[thm]{Corollary}
\newtheorem{eg}[thm]{{Example}}

\theoremstyle{remark}
\newtheorem{rema}[thm]{Remark}

\newcommand{\ad}{{\mbox{\upshape{ad}}}}
\newcommand{\adl}{{\mbox{\upshape{ad}}_l}}
\newcommand{\adr}{{\mbox{\upshape{ad}}_r}}

\newcommand{\adrs}{{\mbox{\upshape{ad}}_r^\ast}}

\newcommand{\C}{{\mathbb C}}

\newcommand{\covA}{A_\rf}
\newcommand{\covAcb}{\cA_{\rfb_{21}}}
\newcommand{\covAc}{\cA_\rf}
\newcommand{\covAcp}{\cA_{\rf'}}

\newcommand{\covAcpw}{\cA_{\rf'}^\wedge}
\newcommand{\N}{{\mathbb N}}
\newcommand{\cA}{{\mathcal A}}

\newcommand{\cB}{{\mathcal B}}
\newcommand{\cC}{{\mathcal C}}
\newcommand{\Cc}{{\mathcal C}}
\newcommand{\cD}{{\mathcal D}}

\newcommand{\cR}{{\mathcal R}}
\newcommand{\cU}{{\mathcal U}}
\newcommand{\cop}{\mathrm{cop}}
\newcommand{\cqg}{{\qfield_q[G]}}
\newcommand{\cqof}{{\qfield_q[\mathcal{O}_f]}}

\newcommand{\End}{\mbox{End}}
\newcommand{\fbar}{\bar{f}}

\newcommand{\field}{{\mathbb K}}

\newcommand{\gfrak}{{\mathfrak g}}
\newcommand{\gbar}{\bar{g}}

\newcommand{\hfrak}{{\mathfrak h}}
\newcommand{\Hom}{{\mathrm{Hom}}}

\newcommand{\id}{{\mbox{id}}}

\newcommand{\kow}{{\varDelta}}

\newcommand{\ltil}{\tilde{l}}

\newcommand{\oF}{{\overline{F}}}

\newcommand{\op}{\mathrm{op}}
\newcommand{\ot}{\otimes}

\newcommand{\Ppiplus}{P^+}

\newcommand{\qfield}{k}
\newcommand{\rf}{{\bf{r}}}
\newcommand{\rfb}{{\bf{\bar{r}}}}

\newcommand{\rp}{{\bf{r}'}}
\newcommand{\rk}{\mathrm{rank}}
\newcommand{\rh}{\hat{R}}

\newcommand{\Q}{\mathbb Q}

\newcommand{\rank}{\mathrm{rank}}
\newcommand{\rep}{\mathrm{Rep}}

\newcommand{\sform}{{\bf{s}}}
\newcommand{\slfrak}{{\mathfrak{sl}}}

\newcommand{\typeone}{\cC}
\newcommand{\ubar}{\overline{u}}
\newcommand{\Uq}{U}

\newcommand{\uqbp}{{U_q(\mathfrak{b}^+)}}
\newcommand{\uqbm}{{U_q(\mathfrak{b}^-)}}

\newcommand{\uqg}{{U_q(\mathfrak{g})}}
\newcommand{\uqgc}{{\check{U}_q(\mathfrak{g})}}
\newcommand{\tr}{\mathrm{tr}}

\newcommand{\vep}{\varepsilon}
\newcommand{\wght}{\mathrm{wt}}
\newcommand{\wurz}{\pi}

\newcommand{\Z}{{\mathbb Z}}

\title[Reflection equation algebras and centres]{Reflection equation
  algebras, coideal subalgebras, and their centres}
\author{Stefan Kolb}
\address{School of Mathematics and Maxwell Institute for Mathematical Sciences,
The University of Edinburgh, JCMB, King's Buildings, Mayfield Road, Edinburgh EH9 3JZ, UK}
\email{stefan.kolb@ed.ac.uk}

\author{ Jasper V.~ Stokman}
\address{Korteweg-de Vries Institute for Mathematics,
University of Amsterdam, Plantage Muidergracht 24, 1018 TV Amsterdam,
The Netherlands}
\email{j.v.stokman@uva.nl}

\begin{document}

\begin{abstract}
  Reflection equation algebras and related $\uqg$-comodule algebras
  appear in various constructions of quantum homogeneous spaces and
  can be obtained via transmutation or equivalently via twisting by a
  cocycle. In this paper we investigate algebraic and representation
  theoretic properties of such so called `covariantized' algebras, in
  particular concerning their centres, invariants, and
  characters. Generalising M.~Noumi's construction of quantum
  symmetric pairs we define a coideal subalgebra $B_f$ of $\uqg$ for
  each character $f$ of a covariantized algebra.  

  The locally finite part $F_l(\uqg)$ of $\uqg$ with respect to the
  left adjoint action is a special example of a covariantized
  algebra. We show that for each character $f$ of $F_l(\uqg)$ the
  centre $Z(B_f)$ canonically contains the representation ring
  $\rep(\gfrak)$ of the semisimple Lie algebra 
  $\gfrak$. We show moreover that for $\gfrak=\slfrak_n(\C)$ such characters
  can be constructed from any invertible solution of the reflection
  equation and hence we obtain many new explicit realisations of
  $\rep(\slfrak_n(\C))$  inside $U_q(\slfrak_n(\C))$. As an example we
  discuss the solutions of the reflection equation corresponding to the
  Grassmannian manifold $Gr(m,2m)$ of $m$-dimensional subspaces in $\C^{2m}$.
\end{abstract}

\maketitle
\section*{Introduction}
Originating in the quantum inverse scattering method of the Leningrad
school, the theory of quantum groups was to a large extent invented to
provide a unified approach to solutions of the quantum Yang-Baxter
equation \cite{inp-Drinfeld1}. Without spectral parameter, these
solutions are well organised in the structure of a
braided monoidal category for each semisimple, finite
dimensional, complex Lie algebra $\gfrak$. Braided monoidal categories
have well-known applications in low dimensional topology and provide
representations of the Artin braid group.   

If one imposes additional boundary conditions \cite{a-Skl88}, then the quantum
Yang-Baxter equation is joined by the so called reflection equation
which first appeared in factorised scattering on the half line
\cite{a-Cher84}. The notion of a braided monoidal category can be
extended to include solutions of the reflection equation \cite{a-tD98},
\cite{a-tDHO98}, however, examples have so far only been constructed
by hand for $\gfrak=\slfrak_n(\C)$ \cite{a-tD99}. The programme
outlined in \cite{a-tD98} aims at applications to braid groups of
type $B$ and the affine braid group. Closely related to this
programme is the notion of a universal solution of the reflection
equation introduced independently in \cite{a-DKM03}. Again, for
quantised universal enveloping algebras, there is no unified
construction of examples.   

The reflection equation is also at the heart of certain classes
of quantum homogeneous spaces. M.~Noumi, T.~Sugitani, and M.~Dijkhuizen used
explicit solutions of the reflection equation to obtain analogues of
all classical symmetric pairs as coideal subalgebras of quantised
universal enveloping algebras, see e.g.~\cite{a-Noumi96},
\cite{a-NS95}, \cite{a-Dijk96}, 
\cite{a-NDS97}, \cite{a-DS99}. Using $l$-operators for the vector
representation their construction follows the spirit of
\cite{a-FadResTak1}. A unified construction of quantum symmetric pairs, more
along the lines of V.~Drinfeld's realisation of quantised universal
enveloping algebras, was achieved by G.~Letzter
\cite{MSRI-Letzter}. Central elements in Letzter's
coideal subalgebras lead to solutions of the reflection
equation \cite{a-Kolb08}.  

If one considers the entries of a matrix satisfying the reflection equation
as indeterminates then one obtains the so called reflection equation
algebra \cite{a-KuSkl92}. Characters of the reflection equation
algebra are the same as numerical solutions of the reflection
equation. Such characters were used by J.~Donin and A.~Mudrov
\cite{a-DM03a} to obtain quantisations of $GL(n)$-orbits in
$\End(\C^n)$. Related works by the same authors, e.g.~\cite{a-DKM03},
\cite{a-DM03a}, \cite{a-Mud07}, centre around the notion of twisting
by a cocycle which goes back to \cite{a-Drin90} and which transforms
FRT-algebras into reflection equation algebras. Twisting by a cocycle,
in turn, is also at the heart of S.~Majid's theory of transmutation
and covariantized algebras \cite[Section 3]{a-Majid91}, \cite[Section
  4]{a-Majid93}, \cite[7.4]{b-Majid1}. In this theory the reflection
equation algebra occurs as so called braided matrices.  

In the present paper we exhibit relations between the five
theories referred to above, namely 
\begin{enumerate}
  \item Twisting by a cocycle and quantisation via
  characters of twisted algebras, e.g.~\cite{a-DKM03}, \cite{a-DM03a},
  \cite{a-DM03}, \cite{a-Mud07},
  \item Transmutation and covariantized algebras \cite{a-Majid91},
    \cite{a-Majid93}, \cite[7.4]{b-Majid1}, 
  \item Universal cylinder forms \cite{a-tD98}, \cite{a-tDHO98}, \cite{a-tD99},
  \item Construction of quantum symmetric pairs via solutions of the
  reflection equation, e.g.~\cite{a-Noumi96}, \cite{a-NS95}, \cite{a-Dijk96},
\cite{a-NDS97}, \cite{a-DS99},
  \item G.~Letzter's construction of quantum symmetric pairs
  \cite{MSRI-Letzter}, \cite{a-Letzter03}
\end{enumerate}
which, for the most part, have been developed independently of each
other. 

Let $G$ be the connected, simply connected affine algebraic
group corresponding to the finite dimensional, semisimple, complex Lie
algebra $\gfrak$. Recall that the quantised algebra of functions
$\cqg$ on $G$ is 
a coquasitriangular Hopf algebra and let $\rf$ denote its universal
$r$-form. We will work in a setting which is tailored to include
FRT-algebras and the quantised algebra of functions $\cqg$ on
$G$. More explicitly, we consider any coquasitriangular bialgebra
$\cA$ together with a homomorphism $\Psi:\cA\rightarrow \cqg$ of
coquasitriangular bialgebras. In this setting, transmutation coincides
with twisting by a cocycle and universal cylinder forms for $\cA$ are, up to
translation of conventions, the same as characters of the
covariantized algebra $\covAc$ of $\cA$.

By construction, the covariantized algebra $\covAc$ is a right comodule
algebra over the quantum double of $\cA$. In our restricted setting,
however, $\covAc$ also has a left $\uqg$-comodule algebra 
structure. This allows us, for any character $f$ of the covariantized
algebra $\covAc$, to define a left coideal subalgebra $B_f\subseteq \uqg$ in
a straightforward manner. We call $B_f$ the Noumi coideal subalgebra
corresponding to $f$, because if $\cA$ is an FRT-algebra then a
character of $\covAc$ is the same as a solution of the reflection
equation and $B_f$ is a coideal of the type constructed in, say
\cite{a-NDS97}.

In this paper we investigate algebraic properties of both the
covariantized algebra $\covAc$  and the Noumi coideal subalgebra
$B_f$ in some detail. We are in particular interested in results
concerning the centres $Z(\covAc)$ and $Z(B_f)$ as well as characters
on $\covAc$. We show among other results, that for characters of
$\covAc$, which are 
convolution invertible with respect to the coproduct of the bialgebra
$\cA$, the centre $Z(B_f)$ is contained in the locally finite part
$F_r(\uqg)$ of $\uqg$ with respect to the right adjoint action. A
similar result was obtained by G.~Letzter for her quantum symmetric
pairs \cite[Theorem 1.2]{a-Letzter-memoirs}. 
The case when $\cA=\cqg$ is of particular interest. We show that in
this case the centre $Z(B_f)$ naturally contains a realisation of the
representation ring $\rep(\gfrak)$ of $\gfrak$. This result may seem
somewhat surprising. It is well known that $Z(\uqg)$ is isomorphic to
$\rep(\gfrak)$ with respect to the grading of $F_r(\uqg)$, see
e.g.~\cite{a-Baumann98}. Our constructions show that there are many more
natural realisations of $\rep(\gfrak)$ inside $\uqg$.

Finally, we address the question of how to obtain characters of the
covariantized algebra $\covAc$. If $\cA$ is an FRT-algebra then the
method devised in \cite{a-Kolb08} provides solutions of the reflection
equation via suitable central elements in coideal subalgebras of
$\uqg$. Here we give a generalised, streamlined presentation of this
result. In the case $\cA=\cqg$ we observe that $\covAc$
coincides with the left locally finite part $F_l(\uqg)$. It hence remains
to determine all characters of $F_l(\uqg)$. At this point we restrict
to the case $\gfrak=\slfrak_n(\C)$ and prove that any suitably scaled
invertible solution of the reflection equation for the vector
representation of $U_q(\slfrak_n(\C))$ factors to a character
of $F_l(U_q(\slfrak_n(\C)))$. Together with the explicit
classification in \cite{a-Mud02} this determines all characters of
$F_l(U_q(\slfrak_n(\C)))$. It would be desirable to have a
classification of characters of $F_l(\uqg)$ for general $\gfrak$. 

As an example we follow G.~Letzter's setting \cite{a-Letzter03} to discuss
the coideal subalgebra $B^s\subseteq U_q(\slfrak_{2m}(\C))$
corresponding to the Grassmann manifold 
$Gr(m,2m)$ of $m$-dimensional subspaces in $\C^{2m}$. Using the
structure of $Z(B^s)$ known from \cite{a-KL08}, we show that in
this case $Z(B^s)$ is naturally isomorphic to $\rep(\slfrak_n(\C))$.
Moreover, we calculate the solution of the reflection equation for the
element in $Z(B^s)$ corresponding to the vector representation. This
allows us to give the corresponding character of the covariantized
algebra explicitly. It doesn't come as a surprise that this character
bears close resemblance to the solutions of the reflection equation 
considered in \cite{a-NDS97}, \cite{a-tD99}, \cite{a-Mud02}, but we
avoid painstaking translation of conventions. 

We now briefly outline the structure of this paper. In Section
\ref{NotConv} we fix notations and conventions for quantised universal
enveloping algebras and quantised algebras of functions. The main
framework of the paper is outlined in Section \ref{transmutation} in
the setting of S.~Majid's theory of transmutation. We establish the
relation to Drinfeld twists and discuss transmutation of $\cqg$ and
FRT-algebras as examples. In Section \ref{AlgProp} we begin the
investigation of algebraic  properties of the covariantized algebra
$\covAc$. In particular we show that $\covAc$ is a domain if and only
if $\cA$ is a domain, which in turn can be used to identify the centre
$Z(\covAc)$ with the space of $\uqg$-invariants in $\cA$. In
\ref{charProp} we collect properties of characters of $\covAc$ which
allows us in Subsection \ref{CylForms} to identify such characters
with universal cylinder forms. Section \ref{Noumi} is devoted to the
construction and 
investigation of the Noumi coideal subalgebra $B_f$ for a given
character $f$ of $\covAc$. In Subsection \ref{NoumiCentre}, in
particular, we establish the realisation of $\rep(\gfrak)$ inside
$Z(B_f)$ in the case $\cA=\cqg$. Moreover, we show in \ref{locFiness}
that $Z(B_f)\subseteq F_r(\uqg)$. The final Section \ref{ConstrChar}
is devoted to the construction of characters of the covariantized
algebra $\covAc$. First we recall the general construction of
solutions of the reflection equation via central elements of coideal
subalgebras of $\uqg$. From Subsection \ref{characters} onwards we
restrict to the case $\gfrak=\slfrak_n(\C)$. In \ref{characters} we prove that any invertible numerical solution of the reflection equation gives rise to a character of
$F_l(U_q(\slfrak_n(\C)))$. The last two subsections are devoted to the
example $Gr(m,2m)$.

{\it Acknowledgements:} The second author was supported by
the Netherlands Organization for Scientific Research (NWO) in the
VIDI-project ``Symmetry and modularity in exactly solvable
models''. Part of the research was done while the first author
visited the University of Amsterdam for two weeks in June 
2008. This visit was supported under Scheme 4 of the London
Mathematical Society and by the above-mentioned VIDI-project. 

\section{Notations and conventions}\label{NotConv}
In this introductory section we fix notations and conventions
concerning quantised universal enveloping algebras and quantised
algebras of functions. All results stated in the first two subsections are well
known. Main sources of reference are the monographs
\cite{b-Jantzen96}, \cite{b-Joseph}, and \cite{b-KS}. In subsection
\ref{locfinl}, to simplify the presentation of the main
thrust of the paper, we recall some possibly less known results concerning
universal $r$-forms and $l$-functionals.

Let $\Z$ be the integers, $\N_0$ the non-negative integers, $\Q$ the
rational numbers, and $\C$ the complex numbers. 
Throughout this text, for any coalgebra $C$ we denote the counit by
$\vep$ and the coproduct by $\kow$. We make use of Sweedler notation
in the form $\kow(c)=c_{(1)}\ot c_{(2)}$ for any $c\in C$, suppressing
the summation symbol. We write $C^\cop$ to denote the opposite
coalgebra with coproduct $c\mapsto c_{(2)}\ot c_{(1)}$.
Similarly, for any algebra $A$ the symbol $A^\op$ denotes the opposite
algebra with multiplication $a\ot b\mapsto ba$. If $A$ is a bialgebra
then $A^{\op,\cop}$ denotes the bialgebra with both opposite
multiplication and opposite comultiplication. For any Hopf algebra
$H$ we use the symbol $\sigma$ for the antipode. We write
$H^\circ$ to denote the dual Hopf algebra consisting of all matrix
coefficients of finite dimensional representations of $H$. 


\subsection{The quantised universal enveloping algebra
  $U:=\uqgc$}\label{Usection}
   Let $\gfrak$ be a finite-dimensional complex semisimple Lie algebra of rank $r$ and let
  $\hfrak$ be a fixed Cartan subalgebra of $\gfrak$. Let $\Delta$ denote the root system
  associated with $(\gfrak,\hfrak)$. Choose an ordered basis 
  $\wurz=\{\alpha_1,\dots,\alpha_r\}$ of simple roots for $\Delta$.
  Let $W$ denote the Weyl group associated to the root system $\Delta$
  and let $w_0$ denote the longest element in $W$ with respect to
  $\wurz$. There is a unique $W$-invariant symmetric bilinear form
  $(\cdot,\cdot)$ on $\hfrak^\ast$ such that $(\alpha,\alpha)=2$ for
  all short roots $\alpha\in \Delta$. This form satisfies
  $(\alpha,\alpha)/2\in \{1,2,3\}$ for all $\alpha\in \Delta$.
  We write $Q$ for the root lattice and $P$ for the 
  weight lattice associated to the root system $\Delta$.
  Set $Q^+=\N_0\pi$ and let $P^+$ be the set of dominant integral weights with 
  respect to $\pi$. We will denote the fundamental weights in $P^+$
  by $\omega_1.\dots,\omega_r$.
  Let $\leq$ denote the dominance partial ordering on 
  $\hfrak^*$, so $\mu\leq\gamma$ if $\gamma-\mu\in Q^+$.

  Let $\qfield=\C(q^{1/N})$ denote the field of rational functions in one
  variable $q^{1/N}$ where $N$ has sufficiently many factors such that 
  $(\lambda,\mu)\in \frac{1}{N} \Z$ for all $\lambda,\mu\in P$. One
  may for instance choose $N$ to be the order of $P/Q$.
  We consider here the simply connected quantised universal enveloping algebra 
  $\uqgc$ as the $\qfield$-algebra generated by elements  
  $\{x_i, y_i, \tau(\lambda)\,|\, i=1,\dots,r,\,\lambda\in P\}$ and 
  relations as given for instance in \cite[Section
  3.2.9]{b-Joseph}. In particular the generators satisfy the following relations
  \begin{align*}
    \tau(\lambda) x_j=q^{(\lambda,\alpha_j)}x_j \tau(\lambda),&\qquad
    \tau(\lambda) y_j=q^{-(\lambda,\alpha_j)}y_j \tau(\lambda),\\
    x_iy_j-y_jx_i&=\delta_{ij} (t_i-t_i^{-1})/(q_i-q_i^{-1})
  \end{align*}
  where, following common convention, we write $t_i$ to denote
  $\tau(\alpha_i)$ and $q_i:=q^{(\alpha_i,\alpha_i)/2}$. The
  generators also satisfy the quantum Serre relations. The
  algebra $\Uq$ has the structure of a Hopf algebra with coproduct and
  antipode give by
  \begin{align*}
    \kow \tau(\lambda)&=\tau(\lambda)\ot \tau(\lambda), \qquad \quad \sigma(\tau(\lambda))=\tau(-\lambda),\\
    \kow x_i&=x_i\ot 1 +t_i\ot x_i, \qquad \sigma (x_i)=-t_i^{-1}x_i,\\
    \kow y_i&=y_i\ot t_i^{-1}+ 1\ot y_i, \qquad \sigma (y_i)=-y_it_i.
  \end{align*}
Note that these conventions also coincide with \cite{b-Jantzen96}
up to renaming of the generators and restriction to the subalgebra
generated by $x_i,y_i,t_i,t_i^{-1}$. 

One checks on the generators that the antipode of $\Uq$ satisfies the
relation
\begin{align}\label{sigma2}
  \sigma^2(u)=\tau(-2\rho)u\tau(2\rho)
\end{align}
where $\rho\in Q^+$ denotes the half sum of all positive roots.

We will write $U^+$ and $U^-$ to denote the subalgebra of $\Uq$ generated by
$\{x_1,\dots, x_r\}$ and $\{y_1,\dots, y_r\}$, respectively.  Let $U^0$ be the subalgebra of $\Uq$
spanned by the elements $\{\tau(\lambda)\,|\,\lambda\in P\}$. Moreover,
we will write  $\uqbp$ and $\uqbm$ to denote the subalgebra of $\Uq$ generated by
$\{x_i,\tau(\lambda)\,|\,i=1,\dots,r\mbox{ and } \lambda\in P\}$ and
$\{y_i,\tau(\lambda)\,|\,i=1,\dots,r\mbox{ and } \lambda\in P\}$,
respectively.

For any $\Uq$-module $V$ and any $\lambda\in P$ we call an element
$v\in V$ a weight vector of weight $\lambda$ if
$\tau(\mu)v=q^{(\mu,\lambda)}v$ for all $\mu\in P$, and we write $\wght (v):=\lambda$. 
For any $\alpha\in Q$ let $U_\alpha$ denote the weight space of $\Uq$ of
weight $\alpha$ with respect to the left adjoint action, more precisely
\begin{align*}
  \Uq_\alpha:=\{u\in \Uq\,|\, \tau(\lambda)\,u\,\tau(-\lambda)=q^{(\lambda,\alpha)}u\mbox{ for all }\lambda\in P\}.
\end{align*}
Note that $\Uq=\bigoplus_{\alpha\in Q}\Uq_\alpha$. Define moreover
$\Uq^+_\alpha:=U^+\cap \Uq_\alpha $ and $\Uq^-_\alpha:=U^-\cap \Uq_\alpha$.

For $\lambda\in P^+$ let $V(\lambda)$ be the simple $\Uq$-module
of highest weight $\lambda$. In particular, there is a highest weight
vector $v_\lambda\in V(\lambda)$ of weight $\lambda$ such that $x_i
v_\lambda = 0$ for all $i=1,\dots,r$. As usual, we say that a
$\Uq$-module is of type one if it has a basis consisting of weight
vectors with weights in $P$. We define $\cC$ to be the category of all
finite dimensional $\Uq$-modules of type one. Recall that $\cC$ is a
rigid, monoidal category via the antipode and the coproduct of
$\Uq$, and that $\cC$ is semisimple with simple objects
$V(\lambda)$ for $\lambda\in P^+$. Moreover, it is well know that
$\typeone$ is a braided monoidal  category
\cite[Chapter 7]{b-Jantzen96}, \cite[9.4.7]{b-Joseph}. One hence has a
family of $\Uq$-module isomorphisms $\rh=(\rh_{V,W}:V\ot W\rightarrow
W\ot V)_{V,W\in Ob(\typeone)}$ which is natural in both $V$ and $W$
and which satisfies the hexagon identities \cite[3.18, 3.19]{b-Jantzen96}.
We recall the construction of the braiding $\rh$ for $\cC$ along the
lines of \cite[Chapter 7]{b-Jantzen96} with a slight change of
convention. The reason for our choice of conventions will become
apparent in Remark \ref{R-choice-remark}.
For any $\alpha\in Q^+$ let $\Theta_\alpha\in U^+_\alpha\ot U^-_{-\alpha}$ denote the canonical
element defined at the beginning of \cite[7.1]{b-Jantzen96} up to
flipping the order of tensor factors. For any $V,W\in Ob(\Cc)$ define
$\Theta_{V,W}:V\ot W\rightarrow V\ot W$ by the action of
the formal sum $\Theta=\sum_{\alpha\in Q^+}\Theta_\alpha$. Moreover, define
a $k$-linear isomorphism $f_{V,W}:V\ot W\rightarrow V\ot W$ by $f_{V,W}(v\ot
w)=q^{-(\wght(v),\wght(w))}v\ot w$ for any weight vectors $v\in V$,
$w\in W$. Now set $\rh_{V,W}:=P_{12}\circ \Theta_{V,W}\circ
f_{V,W}:V\ot W \rightarrow W\ot V$ where $P_{12}$ denotes the flip
of tensor factors. It follows from \cite[7.5, 7.8]{b-Jantzen96} that
$\rh$ defines a braiding on $\typeone$. 
\subsection{The quantised algebra of functions $\cqg$}\label{cqg-qtrace}
Let $G$ denote the connected, simply-connected affine algebraic group
with Lie algebra $\gfrak$. We recall the definition of the quantised
algebra of functions on $G$. For any $V\in Ob(\cC)$ and elements $v\in V$, $f\in
V^\ast$ define a linear functional $c_{f,v}:\Uq\rightarrow \qfield$ by
$c_{f,v}(u):=f(uv)$ for all $u\in \Uq$. If $V=V(\lambda)$ then we also
write $c^\lambda_{f,v}=c_{f,v}$ to keep track of the representation
$V(\lambda)$. Let $C^V:=\mbox{span}\{c_{f,v}\,|\, v\in V,\, f\in
V^\ast\}$ denote the linear span of all matrix coefficients $c_{f,v}$
of the representation $V$. As usual we define $\cqg$ as the Hopf
subalgebra of the Hopf dual $\Uq^\circ$ spanned by the matrix
coefficients of all $V\in Ob(\cC)$. 

It is convenient to define a bilinear pairing of Hopf algebras by evaluation
\begin{align}\label{pairing} 
  \langle\cdot,\cdot\rangle: \cqg\times\Uq\rightarrow \qfield,\qquad
  \langle a,u\rangle:=a(u)
\end{align}
for $a\in \cqg$, $u\in \Uq$. The pairing $\langle\cdot,\cdot\rangle$ is
non-degenerate as a consequence of \cite[Proposition
  5.11]{b-Jantzen96}. The quantised algebra of functions $\cqg$ is a
left $\Uq^\cop\ot \Uq$-module algebra via the action
\begin{align}\label{UcopU}
  (X\ot Y)\cdot a=\langle a_{(1)},\sigma(X)\rangle  \langle a_{(3)},Y
  \rangle a_{(2)}.
\end{align}
By construction $\cqg$ has a Peter-Weyl decomposition
\begin{align}\label{PeterWeyl}
  \cqg=\bigoplus_{\lambda\in \Ppiplus} C^{V(\lambda)}
\end{align}
into irreducible $\Uq^\cop\ot \Uq$-modules. 

The braiding $\rh$ of $\cC$ gives rise to the structure of a universal
$r$-form on $\cqg$. We recall the definition of this notion for
the convenience of the reader.
\begin{defi}\cite[10.1.1]{b-KS} A coquasitriangular bialgebra
  $(A,\rf)$ over a field $\field$ is a pair consisting of a bialgebra $A$ over $\field$ and a
  convolution invertible linear map $\rf:A\ot A\rightarrow \field$ which
  satisfies the following relations
  \begin{align}
        \rf(a_{(1)}\ot b_{(1)}) a_{(2)} b_{(2)}&= b_{(1)} a_{(1)} \rf(a_{(2)}\ot b_{(2)}), \label{rh-hom}\\
        \rf(ab\ot c)&=\rf(a\ot c_{(1)})\,\rf(b \ot c_{(2)}),\label{r-mult}\\
        \rf(a\ot bc)&=\rf(a_{(1)}\ot c)\,\rf(a_{(2)} \ot b),\label{r-mult2}
      \end{align}
  for all $a,b,c\in A$. The map $\rf$ is called a universal $r$-form
  for the bialgebra $A$.
\end{defi}
\begin{rema}
  In the following, to shorten notation, we will suppress tensor
  symbols and write $\rf(a,b)$ instead of $\rf(a\ot b)$. For later reference note that
  any universal $r$-form satisfies
  \begin{align}\label{r-vep}
        \rf(a,1)&=\rf(1,a)=\vep(a)\quad \mbox{for all }a\in A.
  \end{align}
  Note, moreover, that if $A$ is a Hopf algebra then 
  \begin{align}
    \rf(\sigma(a),\sigma(b))&=\rf(a,b) \quad\mbox{for all } a,b\in A\label{r-sigma}
  \end{align}
  and the convolution inverse $\rfb$ of $\rf$ is given by
  $\rfb(a,b)=\rf(\sigma(a),b)$ for all $a,b\in A$. Consult \cite[10.1]{b-KS} for more details.
\end{rema}
For any coquasitriangular bialgebra $(A,\rf)$ the linear map 
$\rfb_{21}:A\ot A\rightarrow \field$ defined by
$\rfb_{21}(a,b)=\rfb(b,a)$ gives rise to a coquasitriangular bialgebra 
$(A,\rfb_{21})$. We will write $\rf'$ to denote either of the
universal $r$-forms $\rf$ or $\rfb_{21}$ for $A$. For $A=\cqg$ and
$\field=\qfield$ consider the linear map $\rf:\cqg\ot \cqg\rightarrow
k$ defined by  
\begin{align}\label{rf-def}
  \rf(c_{f,v},c_{g,w})=(g\ot f) (\rh_{V,W}(v\ot w))
\end{align}
if $c_{f,v}\in C^V$  and $c_{g,w}\in C^W$. It follows from the
properties of the braiding $\rh$ of $\cC$ that $(\cqg,\rf)$ is a coquasitriangular
Hopf algebra. We emphasise that all through this paper the notation $\rf$ as universal
$r$-form on $\cqg$ will always stand for the particular choice (\ref{rf-def}) above.
\subsection{Locally finite part and $l$-functionals}\label{locfinl}
As any Hopf algebra, the quantised universal enveloping algebra $\Uq$
is a left and a right module algebra over itself with respect to the
left and right adjoint actions defined by
\begin{align*}
  \adl(u)X:=u_{(1)}X\sigma(u_{(2)}),\qquad \adr(u)X:=\sigma(u_{(1)})X u_{(2)},
\end{align*}
respectively. The left locally finite part $F_l(\Uq)$ and the right locally finite part $F_r(\Uq)$ are defined by
\begin{align*}
  F_l(U)&:=\{u\in\Uq\,|\, \dim((\adl \Uq)u)<\infty\},\\
  F_r(U)&:=\{u\in\Uq\,|\, \dim((\adr \Uq)u)<\infty\}.
\end{align*}
Note that results for locally finite parts can be translated from left to right and vice versa using the formulae
\begin{align}\label{ad-leftright}
  \sigma((\adl u)X)=(\adr \sigma(u))\sigma(X),\qquad
  \sigma((\adr u)X)=(\adl \sigma(u))\sigma(X).
\end{align}
These formulae imply in particular the relation
\begin{align}\label{FrFl}
  F_r(\Uq)=\sigma(F_l(\Uq))=\sigma^2(F_r(\Uq)).
\end{align}
It was shown in \cite[Theorem 4.10]{a-JoLet2}, \cite{a-Caldero93} that the left locally finite part has a direct sum decomposition
\begin{align}\label{Fl-decomp}
  F_l(\Uq)=\bigoplus_{\lambda \in P^+}(\adl \Uq)\tau(-2\lambda)
\end{align}
into finite dimensional $\Uq$-submodules. Using (\ref{FrFl}) and the
relation $\sigma(\tau(\lambda))=\tau(-\lambda)$ one obtains a similar
decomposition for the right locally finite part 
\begin{align}\label{Fr-decomp}
  F_r(\Uq)=\bigoplus_{\lambda \in P^+}(\adr \Uq)\tau(2\lambda).
\end{align}
The decomposition (\ref{Fl-decomp}) of $F_l(\Uq)$ is closely related
to the Peter-Weyl decomposition (\ref{PeterWeyl}) of $\cqg$ via so
called $l$-functionals. Recall that $\cqg^\circ$ denotes the dual Hopf
algebra of $\cqg$ and that we write $\rf'$ to denote either of the
universal $r$-forms $\rf$ or $\rfb_{21}$ on $\cqg$.   
Following for instance \cite[10.1.3]{b-KS} one obtains linear maps
  $l_{\rf'}^+,   l_{\rf'}^-,l_{\rf'},\ltil_{\rf'}:\cqg\rightarrow \cqg^\circ$ defined by
  \begin{align}
    l_{\rf'}^+(a):=\rf'(\cdot,a)&,\quad  l_{\rf'}^-(a):=\rf'(\sigma(a),\cdot), \label{lpmdef}\\
    l_{\rf'}(a):=l_{\rf'}^-(\sigma^{-1}&(a_{(1)})) l_{\rf'}^+(a_{(2)})
          =\rf'(a_{(1)},\cdot) \rf'(\cdot,a_{(2)}),\label{ldef}\\    
    \ltil_{\rf'}(a):=l_{\rf'}^+(a_{(1)})&l_{\rf'}^-(\sigma^{-1}(a_{(2)}))
         =\rf'(\cdot,a_{(1)})\rf'(a_{(2)},\cdot).\label{ltildef}
  \end{align}
for any $a\in \cqg$. We call these maps $l$-operators and the elements
in their image $l$-functionals. Note that the $l$-operators with
respect to the universal $r$-forms $\rf$ and $\rfb_{21}$ are related by
\begin{align}\label{rrelation}
  l^\pm_{\rfb_{21}}=l^\mp_{\rf}.
\end{align}
The $l$-functionals $l_\rp$ satisfy commutation relations similar to
those of the matrix coefficients in $\cqg$. The following
statement is proved in \cite[10.1.3]{b-KS}.
  \begin{lem}
    For any $a,b\in \cqg$ the following relations hold
    \begin{align}
      l_\rp^\pm(ab)&=l_\rp^\pm(b)l_\rp^\pm(a),\label{lpmmult}\\
      l_\rp^\pm(a_{(1)}) l_\rp^\pm(b_{(1)})\rp(a_{(2)},b_{(2)})&=
       \rp(a_{(1)},b_{(1)})l_\rp^\pm(b_{(2)}) l_\rp^\pm(a_{(2)}),\label{lpmrtt}\\
      l_\rp^-(a_{(1)}) l_\rp^+(b_{(1)})\rp(a_{(2)},b_{(2)})&=
       \rp(a_{(1)},b_{(1)})l_\rp^+(b_{(2)}) l_\rp^-(a_{(2)}).\label{lmprtt}
    \end{align}    
  \end{lem} 
Consider $\Uq$ as a subalgebra of $\cqg^\circ$ via the non-degenerate
Hopf pairing (\ref{pairing}).  
\begin{lem}\label{lplusmin}
$l_\rf^{\pm}(\cqg)\subseteq U_q(\mathfrak{b}^\pm)$.
\end{lem}
\begin{proof}
Fix $V\in Ob(\cC)$. Choose $v\in V$ of weight $\mu\in P$ and $f\in
V^\ast$ of weight $-\nu\in P$. In the notations introduced at the end
of subsection \ref{Usection} one obtains from the definition
(\ref{rf-def}) of the universal $r$-form $\rf$ the
relations
\begin{equation}\label{lexplicit}
\begin{split}
  l_\rf^+(c_{f,v})&=\sum_{\alpha\in Q^+}(\id\ot
  c_{f,v})\big(\Theta_\alpha(\tau(-\mu)\ot 1)\big),\\
l_\rf^-(c_{f,v})&=\sum_{\alpha\in Q^+}(\sigma(c_{f,v})\ot\id)
\big(\Theta_\alpha(1\ot\tau(\nu))\big),
\end{split}
\end{equation}
which show that $l_\rf^{\pm}(\cqg)\subseteq U_q(\mathfrak{b}^\pm)$.
\end{proof}
\begin{rema}\label{R-choice-remark}
  The above lemma explains why we prefer to work with conventions
  concerning the braiding slightly different from those in
  \cite{b-Jantzen96}. In our conventions we get $l^+_\rf(a)\in \uqbp$
  for all $a\in \cqg$ while the braiding in Jantzen's book together
  with definition (\ref{lpmdef}) which agrees with \cite[10.1.3]{b-KS}
  would lead to $l^+_\rf(a)\in \uqbm$.
\end{rema}
\begin{rema}\label{tau-as-ltil}
Let $v_\lambda,v_{w_0\lambda}\in V(\lambda)$ be weight vectors of weights
$\lambda$ and $w_0\lambda$ respectively. Fix also $f_\lambda,f_{w_0\lambda}\in V(\lambda)^\ast$
of weight $-\lambda$ and $-w_0\lambda$ such that 
$f_\lambda(v_\lambda)=1=f_{w_0\lambda}(v_{w_0\lambda})$. By definition \eqref{ltildef}
and \eqref{lexplicit} one obtains
\[
\ltil_{\rf}(c_{f_\lambda,v_\lambda}^\lambda)=\tau(-2\lambda),
\qquad \ltil_{\overline{\rf}_{21}}(c_{f_{w_0\lambda},v_{w_0\lambda}}^\lambda)=
\tau(2w_0\lambda).
\]
These formulae are central in the proof of the following proposition.
\end{rema}
Note that the right adjoint action $\adr$ of $\Uq$ on itself
induces a  left action of $\Uq$ on $\cqg$, defined by
\begin{align}\label{adrs-def}
  \adrs(X)c:=\Delta(X)\cdot c=c\circ \adr(X)
\end{align}
for $c\in \cqg$ and $X\in \Uq$. The following proposition is in principle contained in
\cite{a-Caldero93} (cp.~also \cite[Proposition 2]{a-Kolb08}). We sketch
the proof for the convenience of the reader. 
\begin{prop}\label{CalderoProp}
The $l$-functional $\ltil_{\rf^\prime}$ defines an isomorphism of 
left $\Uq$-modules  $\ltil_{\rf^\prime}: \cqg\rightarrow F_l(\Uq)$, i.e.
$\ltil_{\rf^\prime}(\adrs(X)c)=\adl(X)\ltil_{\rf^\prime}(c)$ 
for all $X\in\Uq$ and $c\in\cqg$. Moreover,
\begin{align}\label{rf-corresp}
  \ltil_{\rf^\prime}\big(C^{V(\lambda)}\big)= (\adl \Uq) \tau(-2\lambda^\prime)
\end{align} 
for all $\lambda\in P^+$, where
$\lambda^\prime=\lambda$ if $\rf^\prime=\rf$ and $\lambda^\prime=-w_0\lambda$
if $\rf^\prime=\overline{\rf}_{21}$.
\end{prop}
\begin{proof}
By (\ref{rrelation}) and by Lemma \ref{lplusmin} the
$l$-functional $\ltil_{\rf^\prime}$ defines a $k$-linear map $\ltil_{\rf^\prime}: \cqg\rightarrow \Uq$. 
It follows for instance from \cite[10.1.3, Proposition 11]{b-KS} that
$\ltil_{\rf^\prime}: \cqg\rightarrow \Uq$ is a morphism of left $\Uq$-modules 
(but be aware of the misprint in formula (28) of \cite[10.1.3]{b-KS}),
hence its image is contained in $F_l(\Uq)$. It now follows from Remark
\ref{tau-as-ltil} that $\ltil_{\rf^\prime}\big(C^{V(\lambda)}\big)=
(\adl \Uq) \tau(-2\lambda^\prime)$. Finally, by \cite[Theorem
  3.5]{a-JoLet2} one has $\dim(C^{V(\lambda)})=\dim((\adl \Uq)
\tau(-2\lambda^\prime))$ and hence $\ltil_{\rf^\prime}: \cqg\rightarrow F_l(\Uq)$ is an
isomorphism.
\end{proof}
  
\section{Transmutation}\label{transmutation}
The method of twisting braided bialgebras by a two-cocycle was
introduced by V.~Drinfeld \cite{a-Drin90} and extends to module
algebras over braided bialgebras. It is at the heart of S.~Majid's
construction of covariantized algebras \cite{a-Majid93} and has been
further investigated by J.~Donin and A.~Mudrov (e.g.~\cite{a-DM03},
\cite{a-DKM03}).

Here we recall this construction within the framework of S.~Majid's theory of transmutation.
We then restrict to a setting, outlined in
\ref{framework}, which is tailored to include $\cqg$ and FRT-algebras
\cite{a-FadResTak1}. The corresponding covariantized algebras are the
locally finite part $F_l(\Uq)$ and reflection equation algebras, respectively.
As a first application we use characters of covariantized algebras to
define quantum adjoint orbits in Subsection
\ref{q-conjugacy}. Finally, in \ref{CylFormDef}, we recall the notion of a universal
cylinder form introduced in \cite{a-tDHO98}. It will become clear in
Subsection \ref{CylForms} that this notion is also closely related to
characters of covariantized algebras. 

The material of this section is drawn from various references which
were developed for the most part independently. We feel that beyond
fixing our framework it is beneficial to give a unified presentation of
these results which are scattered across the literature.

\subsection{Transmutation of coquasitriangular bialgebras} \label{coquasiTrans}
In this subsection we recall S.~Majid's notion of transmutation as
introduced in \cite[Section 4]{a-Majid93} for coquasitriangular Hopf
algebras (cp.~also \cite[Section 7.4]{b-Majid1}). Note that Majid
uses the terminology \textit{dual quasitriangular} instead of
coquasitriangular. We follow the presentation in \cite[10.3]{b-KS}
in the setting of coquasitriangular bialgebras. 

Recall \cite[10.3.1]{b-KS} that a coquasitriangular bialgebra
$(A,\rf)$ over a field $\field$ is called
\textit{regular} if $\rf$ is convolution invertible as a $\field$-linear
functional on the tensor product coalgebra $A\ot A^{\cop}$. In
this case let $\sform$ denote the convolution inverse, and note that if
$A$ is a coquasitriangular Hopf algebra then $\sform(a,b)=\rf(a,\sigma(b))$. 
The \textit{quantum double} is a functor from the category of regular
coquasitriangular bialgebras to the category of bialgebras. By
definition \cite[8.2.1, 10.3.1]{b-KS} the quantum double $D(A,\rf)$
associated to the regular coquasitriangular bialgebra $(A,\rf)$
coincides with $A^{\cop}\ot A$ as a coalgebra. The multiplication
of $D(A,\rf)$ is defined by 
\[(b\otimes a)(b^\prime\otimes a^\prime)= \sum
  \rf(a_{(1)},b_{(3)}^\prime)\mathbf{s}(a_{(3)},b_{(1)}^\prime) 
   b_{(2)}^\prime b\otimes a_{(2)}a^\prime
\]
for $a,a^\prime,b,b^\prime\in A$ (cp.~\cite[8.2.1]{b-KS}). Hence the
canonical linear embeddings $A^{\op,\cop}\hookrightarrow D(A,\rf)$,
$a\mapsto a\ot 1$ and $A\hookrightarrow D(A,\rf)$, $a\mapsto 1\ot a$
are bialgebra homomorphisms. Note that
$A$ is a right $D(A,\rf)$-comodule by
\begin{align}\label{rightD}
  a\mapsto a_{(2)}\otimes (a_{(1)}\otimes a_{(3)})\quad
\mbox{for all $a\in A$}.
\end{align}
Observe furthermore that if $A$ is a Hopf algebra then the map
\begin{align}\label{thetaDef}
  \theta (b\otimes a):=\sigma(b)a,\qquad\mbox{for all $a,b\in A$}
\end{align}
defines a surjective morphism $\theta: D(A,\rf)\rightarrow A$ of bialgebras
(cp. \cite[10.3.2]{b-KS}). 

Transmutation defines a new algebra structure on the vector space $A$
such that (\ref{rightD}) becomes a comodule algebra structure. We
collect the main properties of transmutation in the following
proposition.
\begin{propdef}\cite[10.3.1, Proposition 30]{b-KS}\label{covariantized-pd}
  Let $(A,\rf)$ be a regular coquasitriangular bialgebra over $\field$ with unit $1\in
  A$. The vector space $A$ together with the multiplication $A\ot A
  \rightarrow A,\, a\ot b \mapsto a_{\rf} b$ defined by
  \begin{align}
    a_{\rf} b:=&\rf(a_{(1)},b_{(2)})
    \mathbf{s}(a_{(3)},b_{(1)}) a_{(2)}b_{(3)}\label{rrel1}\\
     =&\rf(a_{(2)},b_{(3)})\mathbf{s}(a_{(3)},b_{(1)})b_{(2)}a_{(1)}\label{rrel2}
  \end{align}
  is a $\field$-algebra with unit 1. This algebra is denoted by
  $\covA$ and is called the covariantized algebra associated to
  $(A,\rf)$. The covariantized algebra $\covA$ is a right
  $D(A,\rf)$-comodule algebra with respect to the coaction
  (\ref{rightD}). The multiplication of $A$ can be written in terms of
  the multiplication of $\covA$ as
  \begin{align}
    ab&=\overline{\rf}(a_{(1)},b_{(1)})\rf(a_{(3)},b_{(2)})
    a_{(2)}{}_{\mathbf{r}} b_{(3)}\label{rel1}\\
   &=\rf(b_{(2)},a_{(1)})\overline{\rf}(b_{(3)},a_{(3)})
    b_{(1)}{}_{\mathbf{r}} a_{(2)}\label{rel2} 
  \end{align}
  for $a,b\in A$. 
\end{propdef}

\subsection{Framework}\label{framework} We now outline the general
framework in which we will remain for the rest of this paper with the
exception of Proposition-Definition \ref{TwistPropDef} and Subsection \ref{CylFormDef}.
Throughout this paper $(\cA,\Psi)$ denotes a pair consisting of a
coquasitriangular bialgebra $(\cA,\rf_\cA)$ and a fixed homomorphism
$\Psi:(\cA,\rf_\cA)\rightarrow (\cqg,\rf)$ of coquasitriangular
bialgebras where $\rf$ is the universal $r$-form defined by
\eqref{rf-def}. Note that in this setting $\rf_\cA=\rf\circ(\Psi\ot
\Psi)$. We will hence, by slight abuse of notation, suppress the
subscript $\cA$ and denote the universal $r$-form of $\cA$ also by
$\rf$. Similarly, we will write $(\cA,\rfb_{21})$ instead of
$(\cA,\rfb_{21}\circ(\Psi\ot \Psi))$. Note that both $(\cA,\rf)$ and
$(\cA,\rfb_{21})$ are regular coquasitriangular bialgebras and we can
hence form the covariantized algebras $\covAc$ and $\covAcb$
as outlined in Proposition-Definition \ref{covariantized-pd}.

The morphism $\Psi$ allows us to transfer key structures of $\cqg$ to
$\cA$. These structures involve the Hopf pairing (\ref{pairing}) and the
antipode of $\cqg$.
\begin{lem}\label{UU-mod}
    The algebra $\cA$ is a left $\Uq^\cop\ot \Uq$-module algebra
    with respect to the action defined by
    \begin{align}\label{acc}
      (X\otimes Y)\cdot a=\langle \Psi(a_{(1)}),\sigma(X)\rangle
      \langle \Psi(a_{(3)}),Y\rangle a_{(2)}
    \end{align}
    for $X,Y\in \Uq$ and $a\in\cA$. With respect to this action $\cA$ is a direct sum of finite dimensional,
simple, type one $\Uq\ot\Uq$-modules.
\end{lem}
\begin{proof}
  The first claim is verified in the same way as one proves that
  (\ref{UcopU}) defines a $\Uq^\cop\ot \Uq$-module algebra structure on
  $\cqg$. To verify the second statement observe that $\cA$ is a
  locally finite $\Uq\ot\Uq$-module since any  element $a\in \cA$ is
  contained in a finite dimensional subcoalgebra of $\cA$
  \cite[Theorem 2.2.1]{b-Sweedler}. Moreover, $\mathcal{A}$ is a type
  one $\Uq\ot\Uq$-module  since $\cqg$ consists of the matrix
  coefficients of type one $\Uq$-modules.
\end{proof}
Recall that all through this paper we use the notation $\rf'$ to
denote either $\rf$ as defined by (\ref{rf-def}) or $\rfb_{21}$. In
the next lemma we transfer the left $\Uq$-module structure of $\cqg$
defined by (\ref{adrs-def}) to the covariantized algebra $\covAcp$.
\begin{lem}\label{transferAr}
  {\upshape (\bf i)} The covariantized algebra $\covAcp$  is a right $\cqg$-comodule
  algebra with the coaction $\delta:\covAcp\rightarrow \covAcp\ot
  \cqg$ defined by 
  \begin{align} \label{Ar-coact}
    \delta(a)=a_{(2)}\otimes \sigma(\Psi(a_{(1)}))\Psi(a_{(3)}).
  \end{align}
  {\upshape (\bf ii)} The covariantized algebra $\covAcp$ is a left $\Uq$-module algebra
  by
  \begin{align}\label{ad-def}
    \adrs(X)a:=\Delta(X)\cdot a= \langle
    \sigma(\Psi(a_{(1)})) \Psi(a_{(3)}),X \rangle a_{(2)} 
  \end{align}
  for $X\in \Uq$ and $a\in\cA$.
\end{lem}
\begin{proof}
  The bialgebra map $\Psi$ allows us the push the right $D(\cA,\rf_{\cA}^\prime)$-comodule algebra structure
on $\covAcp$ to a right $D(\cqg,\rf^\prime)$-comodule algebra structure 
\begin{align}\label{D-comod}  
  a\mapsto a_{(2)}\otimes (\Psi(a_{(1)})\otimes\Psi(a_{(3)}))
\end{align}
on $\covAcp$. Composing this map with the bialgebra homomorphism (\ref{thetaDef}) for $A=\cqg$
one obtains the right $\cqg$-comodule algebra structure $\delta$ on
$\covAcp$. This proves {\bf(i)} and {\bf(ii)} follows immediately from
the fact that (\ref{pairing}) is a Hopf pairing.
\end{proof}
At the beginning of this subsection we agreed to write $\rf(a,b)$ for
$a,b\in \cA$ instead of $\rf_\cA (a,b)=\rf(\Psi(a),\Psi(b))$.
To simplify notations in later calculations, we now push this convention
further and suppress the symbol $\Psi$ as far as possible.

\medskip

\noindent{\bf Convention: } We will allow elements of $\cA$ inside the
arguments of the universal $r$-form $\rf$ for $\cqg$, the
$l$-operators $l_{\rp}^+,l_\rp^-,l_\rp,\ltil_\rp$, and inside
the first entry of the pairing (\ref{pairing}). Whenever an element
$a\in \cA$ occurs in one of these places the homomorphism $\Psi$ has
to be applied to $a$ first. Similarly, we write $\sigma(a)$ instead of
$\sigma(\Psi(a))$ for any $a\in \cA$. For example,
$\ltil_{\rf^\prime}(\sigma(a)b)$ with $a,b\in \cA$, should be read as
$\ltil_{\rf^\prime}(\sigma(\Psi(a))\Psi(b))$. 

\medskip
It is important to remember that $\cA$ is only a bialgebra and not
necessarily a Hopf algebra. Note that in our conventions the
multiplication in the covariantized algebra $\cA_\rp$  can be
written as
\begin{align}
  a_\rp b&=\rp(a_{(1)},b_{(2)})\rp(a_{(3)},\sigma(b_{(1)})) a_{(2)}b_{(3)}\label{rrrel1}\\
  &=\rp(a_{(2)},b_{(3)})\rp(a_{(3)},\sigma(b_{(1)}))b_{(2)}a_{(1)} \label{rrrel2}
\end{align}
for $a,b\in\cA$, while on the other hand,
\begin{align}
ab&=\rp(\sigma(a_{(1)}),b_{(1)})\rp(a_{(3)},b_{(2)})a_{(2)}{}_\rp b_{(3)}\label{rel01}\\
&=\rp(b_{(2)},a_{(1)})\rp(\sigma(b_{(3)}),a_{(3)})b_{(1)} {}_\rp a_{(2)}\label{rel02}
\end{align}
for $a,b\in\mathcal{A}$.

We end this subsection with a discussion of the image of the map
$\Psi$. Note that $\Psi(\cA)$ is a subbialgebra of $\cqg$. As we
couldn't pinpoint the classification of subbialgebras of $\cqg$ in the
literature we provide the following proposition for the convenience of
the reader. For simplicity we assume that $\gfrak$ is a simple
Lie algebra. 

Recall that for any lattice $L\subset \hfrak^\ast$ with $Q\subseteq 
L\subseteq P$ there exists a uniquely determined affine algebraic group
$G_L$ with Lie algebra $\gfrak$ and fundamental group $P/L$
\cite[Chapter XI]{b-Humphreys75}. We define $\qfield_q[G_L]$ to be the
subalgebra of $\cqg$ generated by all matrix coefficients of type one
representations of $\Uq$ with highest weight in $L$.
\begin{prop}\label{cqGLprop}
  Assume that $\gfrak$ is simple and let $\cB\neq \qfield 1$ be a
  subbialgebra of $\qfield_q[G]$. Then $\cB=\qfield_q[G_L]$ for some
  lattice $L\subset \hfrak^\ast$ with  $Q\subseteq L\subseteq P$. 
\end{prop}
\begin{proof}
 It follows from the fact that the coalgebra $\cqg$ is cosemisimple that
 \begin{align}\label{A=}
   \cB=\bigoplus_{\lambda\in M} C^{V(\lambda)}
 \end{align}
 for some subset $M\subseteq P^+$. For any $V\in Ob(\cC)$ of dimension $N$ the $\Uq$-module
 $V^{\ot(N-1)}$ contains a copy of the dual representation
 $V^\ast$. This implies that $\cB$ is a Hopf subalgebra of
 $\cqg$. Hence the span of all matrix coefficients of $\gfrak$-modules
 with highest weight in $M$ is a Hopf subalgebra of the coordinate
 ring $\C[G]$ of the affine algebraic group $G$.
 It now follows from the classification of semisimple affine algebraic groups \cite[Chapter
  XI]{b-Humphreys75} that $M=L\cap P^+$ for some lattice $L\subset
 \hfrak^\ast$ with $Q\subseteq L\subseteq P$. 
\end{proof} 
\subsection{Transmutation and Drinfeld twists}\label{twistingprop}
We give an alternative construction of the covariantized algebra
$\cA_\rp$ using Drinfeld twists \cite{a-Drin90}. While this is not
indispensable for the understanding of this paper, it allows us to relate later
results to the work of J.~Donin and A.~Mudrov (e.g.~\cite{a-DM03},
\cite{a-DKM03}, \cite{a-Mud07}). 
For the convenience of the reader we recall the construction of
twisted algebras in the following proposition. We denote
any bialgebra $H$ with unit $1_H$,
multiplication $\mu:H\otimes H\rightarrow H$, counit $\vep$, and coproduct
$\kow$ by $(H,1_H,\mu,\vep,\kow)$. Moreover, we freely make use of the well
known leg notation: For any unital algebra $B$ and any $i,j,n\in \N$
let $\phi_{ij}:B\ot B\rightarrow B^{\ot n}$ denote the algebra
homomorphism defined by $\phi_{ij}(x\ot y)=1\ot\dots \ot x \ot \dots\ot
y\ot \dots\ot 1$ with $x$ and $y$ in the $i$-th and $j$-th position,
respectively. For any element $F\in B\ot B$ we write
$F_{ij}:=\phi_{ij}(F)$ with $n$ being understood from the context.

\begin{propdef}\label{TwistPropDef}
   {\upshape \bf (i)} \cite{a-Drin90} Let $(H,1_H,\mu,\vep,\kow)$ be a bialgebra and let
       $F\in H\otimes H$ be an invertible element such that
       \begin{equation}\label{gauge}
         \begin{split}
           (\kow\otimes\textup{Id})(F)F_{12}&=(\textup{Id}\otimes\kow)(F)F_{23},\\
           (\epsilon\otimes\textup{Id})(F)&=1_H=(\textup{Id}\otimes\epsilon)(F).
         \end{split}
       \end{equation}
       For all $x\in H$ define 
      \[\kow_F(x)=F^{-1}\kow(x)F.
      \]
      Then  $H_F:=(H,1_H,\mu,\vep,\kow_F)$ is a bialgebra. We call $F$ a
      twist for $H$.\\
    {\upshape \bf (ii)} \cite{a-DM03}, \cite[Proposition 7]{a-DKM03} Let $H$ be a bialgebra and $F\in H\ot H$ a
      twist for $H$. Let $(A,1_A,m)$ be a unital left $H$-module algebra
      with multiplication  $m:A\ot A\rightarrow A$.
      Define a linear map
      \begin{align*}
        m_F:A\ot A\rightarrow A,\qquad m_F(a\otimes b)=m(F(a\otimes b)).
      \end{align*}  
      Then $A_F:=(A,1_A,m_F)$ is an algebra. The left $H$-module structure on $A$
      turns $A_F$ into a left $H_F$-module algebra. To simplify
      notation we define $a_F b:=m_F(a\ot b)$. \\
    {\upshape \bf (iii)} \cite{a-DM03}, \cite[Proposition 3]{a-DKM03} Let  $(\cU,1,\mu,\vep,\kow,\cR)$
      be a braided bialgebra with universal $R$-matrix $\cR$ and $H:=\cU^{\mathrm{cop}}\ot \cU$ the
      product bialgebra. Then the two elements
      \begin{align} \label{FFb}
         F:=\cR_{13} \cR_{23}\in H \ot H, \qquad
         \oF:=\cR_{24}^{-1}\cR_{14}^{-1}\in H\ot H 
      \end{align}
      are twists for $H$. 
\end{propdef} 
We now return to the setting of Subsection \ref{framework}. We want to apply
the above proposition for $\cU=\Uq$ and $A=\cA$. Here $\cR$ is given
formally via the braiding of $\cC$ as the family $\cR=\{P_{12}\circ \rh_{V,W}:V\ot
W\rightarrow V\ot W\,|\, V,W \in Ob(\cC)\}$ of linear maps. Hence all formulae involving $\cR$ have to be
interpreted as actions on suitable tensor products of finite
dimensional $\Uq$-modules. Recall from Lemma \ref{UU-mod} that $\cA$
is a $\Uq^\cop\ot \Uq$-module algebra. The algebras $(U^\cop\ot U)_F$ and
$(U^\cop\ot U)_\oF$ do not exist, but in view of the second half of Lemma \ref{UU-mod} it is
still possible to form the algebras $\cA_F$ and $\cA_\oF$.
\begin{lem}\label{gaugelem}
The identity map of $\cA$ defines isomorphisms $ \covAc\simeq \cA_F$ and $\covAcb\simeq 
(\cA_{\oF})^{op}$ of algebras.
\end{lem}
\begin{proof}
By the definition of the universal $r$-form $\rf$ 
and by the explicit expressions \eqref{FFb} 
of $F$ and $\oF$ we have
\begin{align}
a_Fb &=\rf(a_{(3)},\sigma(b_{(1)}))
\rf(a_{(1)},b_{(2)})a_{(2)}b_{(3)},\label{aFb1}\\
a_{\oF}b &=\rf(\sigma(a_{(3)}),b_{(2)})
         \rf(\sigma^2(a_{(1)}),b_{(3)})a_{(2)}b_{(1)} \label{aoFb1}
\end{align}
for $a,b\in\cA$. Comparing with \eqref{rrrel1} and \eqref{rrrel2}
we conclude that $a_Fb=a_{\rf}b$ and $b_{\oF}a=a_{\overline{\rf}_{21}}b$
for all $a,b\in\cA$.
\end{proof}
\begin{rema}
  In our setting the analogue of the $H_F$-module algebra structure from
  Proposition-Definition \ref{twistingprop}.(ii) is provided by the right
  $D(\cqg,\rf)$-comodule algebra structure \eqref{D-comod} of $\covAcp$.
\end{rema}
\subsection{Transmutation of $\cqg$}
  The first important class of examples for the general setting
  outlined in the Subsection \ref{framework} consists of the
  bialgebras $\cA:=\qfield_q[G_L]$ for a lattice $L\subset \hfrak^\ast$ with $Q\subseteq
  L\subseteq P$. In this case the map $\Psi$ is just the embedding. In
  this subsection we discuss the special case where $\cA=\cqg$ with
  universal $r$-form $\rf$. Recall from Lemma \ref{transferAr}.(ii) that
  $\cqg_\rp$ is a left $U$-module algebra with respect to the action
  $\adrs$. The left locally finite part $F_l(\Uq)$, on the other hand,
  is a left $U$-module algebra with respect to the left adjoint
  action $\adl$.
  \begin{prop}\label{lr-iso}
    The $l$-functional $\ltil_\rp:\cqg_\rp\rightarrow F_l(\Uq)$
    defines an isomorphism of left $\Uq$-module algebras. 
  \end{prop}
\begin{proof}
  We have already seen in Proposition \ref{CalderoProp} that
  $\ltil_\rp:\cqg\rightarrow F_l(\Uq)$ is an isomorphism of left
  $\Uq$-modules. By \cite[10.1.3 (30)]{b-KS} the map
  $\ltil_\rp:\cqg_\rp\rightarrow F_l(\Uq)$ is also an algebra
  homomorphism.  
\end{proof}
\begin{rema}\label{STSPoisson}
It was pointed out in \cite{a-Mud07} that in the case $\cA=\cqg$ the
twisted algebra $\cA_F=\cA_\rf$ is a quantisation of the coordinate
ring $\C[G]$ of the affine algebraic group $G$ with respect to the so called
Semenov-Tian-Shansky Poisson bracket. We may hence view the locally
finite part $F_l(\Uq)$ as a quantisation of the coordinate ring of $G$
considered as a $G$-space under the action by conjugation. 

Moreover, from this perspective the analogue of the
Kostant-Richardson theorem obtained in \cite{a-Mud07} in the $h$-adic
setting corresponds to the separation of variables theorem obtained in
\cite{a-JoLet2} for $\Uq$.
\end{rema}

\subsection{FRT algebras and reflection equation algebras} \label{FRT-RE}
  We now discuss the second important class of examples for the
  general setting outlined in Subsection \ref{framework}.
  For any $V\in Ob(\cC)$ the braiding of $\cC$ defined in
  Subsection \ref{Usection} provides a linear automorphism
  $R_{V,V}:=P_{12}\circ\rh_{V,V}$ of $V\ot V$, where $P_{12}$
  denotes the flip of tensor factors. This automorphism satisfies the
  quantum Yang-Baxter equation and hence can
  be used to perform the FRT-construction \cite{a-FadResTak1},
 \cite[9.1]{b-KS}. One obtains a coquasitriangular bialgebra
  $\cA(R_{V,V})$ which is generated as an algebra by the linear space
  $V^\ast\ot V$. In this case the map $\Psi$ is induced by the
  canonical map $f\ot v\mapsto c_{f,v}$ of $V^\ast\ot V$ onto the
  linear space of matrix coefficients $C^V\subset \cqg$.  
    
  To make this construction more explicit assume $\dim(V)=N$ and choose a basis
  $\{v_1,\dots,v_N\}$ of $V$ with dual basis $\{f_1,\dots,f_N\}$. We
  denote the generator of $\cA(R_{V,V})$ corresponding to
  $f_i\ot v_j$ by $t^i_j$. Let $R$ denote the $(N^2\times N^2)$-matrix corresponding to
  $R_{V,V}$ with respect to the chosen bases and consider the
  generators $t^i_j$ as entries of an $(N\times N)$-matrix $T$. Then the defining relations of
  $\cA(R_{V,V})$ may be written as
    \begin{align}\label{FRT-rel}
      R T_1 T_2 = T_2  T_1 R
    \end{align}
  where $T_1=T\ot \id$, $T_2=\id \ot T$ are  $(N^2\times
  N^2)$-matrices with entries in $\cA(R_{V,V})$. It is well known \cite[9.1.1, 10.1.2]{b-KS}
  that  $\cA(R_{V,V})$ has the structure of a coquasitriangular
  bialgebra with coproduct
  \begin{align*}
    \kow(t^i_j)=\sum_{m=1}^N t^i_m\ot t^m_j 
  \end{align*}
  and universal $r$-form defined by
  $\rf_{\cA(R_{V,V})}(t^i_j,t^m_n)=\rf(c_{f_i,v_j},c_{f_m,v_n})$. In
  this case the homomorphism of coquasitriangular bialgebras
  $\Psi:\cA(R_{V,V})\rightarrow \cqg$ is given by $\Psi(t^i_j)=c_{f_i,v_j}$.
\begin{prop}\label{FRTRE}\cite[p.~334]{b-Majid1}, \cite[10.3.1, Example 18]{b-KS}, 
  \cite[Prop.~ 4.11]{a-DM03} 
  In the above setup, the covariantized algebra $\cA(R_{V,V})_\rp$ is
  isomorphic to the unital $\qfield$-algebra with generators $s^i_j$,
  for $1\le i,j\le N$, and relations 
  \begin{align}\label{refleqn}
    S_2R_{21}^\prime S_1R_{12}^\prime=R_{21}^\prime S_1 R_{12}^\prime S_2,
  \end{align}
  where $S$ is the $N\times N$ matrix with entries $s^i_j$ and $R'=R$
  if $\rf'=\rf$ while $R'=R^{-1}_{21}$ if $\rp=\rfb_{21}$.
  \end{prop}
\begin{proof}
  By \eqref{rel01} and \eqref{rel02} the generators $s^i_j:=t^i_j$
  of $\cA(R_{V,V})_\rp$ satisfy the relations \eqref{refleqn}.
  The fact that \eqref{refleqn} are the defining relations for
  $\cA(R_{V,V})_\rp$ in terms of the algebraic generators
  $s^i_j:=t^i_j$ follows along the lines of say \cite[Appendix A]{a-Mud07}. 
\end{proof}
The covariantized algebra $\cA(R_{V,V})_\rp$ is called the reflection equation
algebra associated to $R^\prime$. It is also referred to as the
braided matrix algebra associated to $R^\prime$.  

\subsection{Quantum adjoint orbits}\label{q-conjugacy}
For any unital $\qfield$-algebra $\cB$ we denote by $\cB^\wedge$ the
set of characters of $\cB$. More explicitly, $\cB^\wedge$ is the set
of nonzero algebra homomorphisms $f:\cB\rightarrow \qfield$. Note in
particular that a character $f:\cB\rightarrow \qfield$ automatically
preserves the units.  
 
For any character $f\in \covAcpw$ the coaction $\delta$ defined by
  (\ref{Ar-coact}) allows us to construct an algebra homomorphism
  \begin{align*}
    \delta_f:\covAcp\rightarrow \cqg,\qquad \delta_f(a):=(f\ot \id)\delta(a).
  \end{align*}
We define 
\begin{align*}
  \cqof:=\delta_f(\covAcp)
\end{align*}  
and call $\cqof$ the quantum adjoint orbit associated to the character $f$.
\begin{lem}
  For any character $f\in\covAcpw$ the following hold:
  \begin{enumerate}
    \item $\cqof$ is a right coideal subalgebra of $\cqg$. 
     \item
    $\cqof=\{\sigma(\Psi(a_{(1)}))f(a_{(2)})\Psi(a_{(3)})\,|\,a\in
    \cA\}$.
   \end{enumerate}
\end{lem} 
\begin{proof}
The first statement of the lemma holds because
$\delta_f:\covAcp\rightarrow \cqg$ is a homomorphism of right $\cqg$-comodule
algebras. The second statement holds by definition of $\delta_f$.
\end{proof}  
\begin{rema}
Recall from Remark \ref{STSPoisson} that if $\cA=\cqg$ then $\covAcp$ is a
quantum analogue of the algebra of functions on $G$ considered as a
$G$-space with respect to the adjoint action. Hence in this case
$\cqof$ is a quantum analogue of the algebra of functions on
the conjugacy class of the classical point corresponding to the
character $f\in\covAcpw$. If, on the other hand, $\cA=\cA(R_{V,V})$ is
an FRT algebra as in Subsection \ref{FRT-RE}, then $\cqof$ is a quantum
analogue of the functions on the orbit of an element in $\End(V)$ under conjugation by $G$.
\end{rema}
\begin{rema}
  Quantum adjoint orbits for characters of the reflection equation
  algebra were defined and investigated in \cite{a-DM03a}. In that
  paper, for $G=SL(N)$ and the vector representation, these algebras are
  given explicitly in terms of generators and relations based on the
  classification in \cite{a-Mud02}. Moreover, their semiclassical limits are determined.
\end{rema}

\subsection{Universal cylinder forms}\label{CylFormDef}
In this subsection we give another motivation for the investigation
and construction of characters of covariantized algebras. 
The notion of a universal cylinder form for a coquasitriangular bialgebra
$(A,\rf_A)$ was introduced in \cite{a-tDHO98}. It appeared as a necessary
ingredient in order to find representations of the braid group of type
$B_n$ coming from quantum groups. Examples were constructed in
\cite{a-tDHO98} for $A=\qfield_q[SL(2)]$ and in \cite{a-tD99} for
$A=\qfield_q[SL(N)]$. We recall the main definition from \cite[(1.1)]{a-tDHO98}.
\begin{defi}
  Let $(A,\rf_A)$ be a coquasitriangular bialgebra over a field $\field$. A
  linear functional $f:A\rightarrow \field$ is called a universal cylinder form for
  $(A,\rf_A)$ if it is convolution invertible and satisfies
  \begin{align}
    f(ab)&= f(a_{(1)}) \rf_A(b_{(1)},a_{(2)})f(b_{(2)})\rf_A(a_{(3)},b_{(3)})\label{cyleq1}\\
         &= \rf_A(b_{(1)},a_{(1)})f(b_{(2)})\rf_A (a_{(2)},b_{(3)}) f(a_{(3)})\label{cyleq2}
  \end{align}
  for all $a,b\in A$. We denote by $CF(A,\rf_A)$ the set of universal cylinder forms
  for $(A,\rf_A)$.
\end{defi}
We will explain in subsection \ref{CylForms} that in our setting a
universal cylinder form is the same as a character of a covariantized
algebra. In view of the similarity of Relations (\ref{rel01}),
(\ref{rel02}) and Relations (\ref{cyleq1}), (\ref{cyleq2}), this is nearly
obvious, however, conventions in \cite{a-tDHO98} slightly differ from ours.

\section{Algebraic properties of covariantized algebras}\label{AlgProp}
We now begin the general investigation of the covariantized algebra
$\covAcp$ in the framework outlined in Subsection \ref{framework}. In
this section we collect algebraic properties of $\covAcp$ which are
obtained via the $\Uq^\cop\ot \Uq$-module algebra structure \eqref{acc} of $\cA$
and the $\Uq$-module algebra structure $\adrs$ of $\covAcp$ defined in
\eqref{ad-def}. In particular, we obtain results about zero divisors in
$\covAcp$ and relate the centre of $\covAcp$ to the space of
$\adrs(\Uq)$-invariant elements. Moreover, we collect properties of
characters of covariantized algebras and establish the relation to
universal cylinder forms.
\subsection{Zero divisors}

In this subsection we show that the covariantized algebra $\covAcp$
is a domain if and only if the algebra $\cA$ is a domain.
We start with the following convenient way to express the multiplication
\eqref{rrrel1}--\eqref{rel02} of $\covAcp$ in terms of the adjoint
action \eqref{ad-def} of the $l$-functionals $l^{\pm}_{\rf^\prime}$ on $\cA$.
\begin{lem}\label{compact-aFb}
  For all $a,b\in \cA$ the following relations hold:
  \begin{align}   
    a_{\rf^\prime}b &= \big(\adrs\,[l^+_{\rf^\prime}
    (\sigma(b_{(1)}))]\,a \big)b_{(2)}= \big(
    \adrs \,[l^-_{\rf^\prime}(\sigma^{-1}(a_{(2)}))]\,b \big)a_{(1)},\label{aFbL1}\\
    ab   &= b_{(1)}{_{\rf^\prime}} \big(\adrs \,[l^-_{\rf^\prime}(b_{(2)})]\,a \big)
          = \big(\adrs \,[l^+_{\rf^\prime}(b_{(1)})]\,a \big)_{\rf^\prime}b_{(2)}.
    \label{aFbL2}   
  \end{align}
\end{lem}
Recall that a weight vector $v$ in a $\Uq$-module is called a highest
weight vector, if $Xv=0$ for all $X\in \Uq^+$. For highest weight
vectors with respect to the action (\ref{ad-def}) the gauged
multiplication takes a very simple form.
\begin{lem}\label{ahwFb}
  Let $a\in \cA$ be a highest weight vector of weight $\lambda\in
  \Ppiplus$ with respect to the action $\adrs$. Then the following
  relations hold
  \begin{align*}
    a_{\rf}b&=a[(\tau(-\lambda)\ot 1)\cdot b],\\
    b_{\rfb_{21}}a&=a[(1\ot\tau(\lambda))\cdot b]
\end{align*} 
for all $b\in \cA$.
\end{lem}
\begin{proof}
  It follows from \eqref{lexplicit} that
  \[\adrs(l_{\rf}^+(b))a=\langle b,\tau(-\lambda)\rangle a
  \]
  for all $b\in\cA$. If $\rf^\prime=\rf$, then \eqref{aFbL1} implies
  \[a_{\rf}b=(\adrs[l_{\rf}^+(\sigma(b_{(1)}))]a)b_{(2)}=
    a\langle \sigma(b_{(1)}),\tau(-\lambda)\rangle b_{(2)}=
    a[(\tau(-\lambda)\ot 1)\cdot b].
  \]
  If $\rf^\prime=\overline{\rf}_{21}$, then \eqref{rrelation} and
  \eqref{aFbL1} give
  \[b_{\overline{\rf}_{21}}a=(\adrs[l_{\rf}^+(\sigma^{-1}(b_{(2)}))]a)b_{(1)}=
    a \langle \sigma^{-1}(b_{(2)}),\tau(-\lambda)\rangle b_{(1)}=
    a[(1\ot\tau(\lambda))\cdot b].
  \]
\end{proof}
\begin{prop}\label{AdAFd}
  The algebra $\cA$ is a domain if and only if $\covAcp$ is a domain.
\end{prop}
\begin{proof}
  Suppose that $\covAcp$ has zero divisors. Recall that $\covAcp$ is a type
  one, locally finite $\Uq$-module algebra with respect to the action
  $\adrs$. Hence there exist $\adrs(\Uq)$ highest weight vectors $a,b\in
  \covAcp$ satisfying $a_{\rp}b=0$ (cp.~the argument given in the
  proof of \cite[9.1.9(1)]{b-Joseph}). By Lemma \ref{ahwFb} we
  conclude that $a$ or $b$ is a zero divisor in $\cA$.

  Conversely, suppose that $\cA$ has zero divisors. Considering this time
  $\cA$ as the type one, locally finite $\Uq^{\cop}\ot\Uq$-module algebra
  with respect to the action \eqref{acc}, there exist $\Uq\ot\Uq$ 
  highest weight vectors
  $a,b\in\cA$ satisfying $ab=0$. Note that in this case
  $(\tau(-\lambda)\ot 1)\cdot b$ and $(1\ot \tau(\lambda))\cdot b$ are
  nonzero scalar multiples of $b$ for any $\lambda\in P$. Moreover, $a$ and $b$ are also highest weight
  vectors for the action $\adrs$ and hence Lemma \ref{ahwFb} implies that
  $a_\rf b=0$ and $b_{\rfb_{21}}a=0$.
\end{proof}

\subsection{Invariants and centres}
For any  left $\Uq$-module $\cB$ define the space of invariants
$\cB^{inv}$ by
\begin{align*}
  \cB^{inv}:=\{ b\in \cB\,|\, Xb=\vep(X)b \,\mbox{ for all $X\in \Uq$}\}.
\end{align*}
Note that if $\cB$ is a $\Uq$-module algebra, then $\cB^{inv}\subseteq
\cB$ is a subalgebra. We remain in the setting of Subsection
\ref{framework}. In the present subsection we describe properties of the subspace $\cA^{inv}$ of
invariants of $\cA$ with respect to the $\Uq$-action $\adrs$ defined by (\ref{ad-def}). Note first
that in terms of the coaction $\delta$ defined
by (\ref{Ar-coact}) one has
\begin{align*}
  \cA^{inv}=\{a\in \cA\,|\, \delta(a)=a\ot 1\}.
\end{align*}
As $\cA$ and the covariantized algebra $\covAcp$ coincide as vector
spaces, we may consider $\cA^{inv}$ alternatively as a subspace of
$\cA$ or $\covAcp$. By Lemma \ref{transferAr}(ii) $\covAcp$ is a
$\Uq$-module algebra with respect to the action $\adrs$ and hence
$\cA^{inv}\subseteq \covAcp$ is a subalgebra. 
In the following proposition and in
Proposition \ref{ZAF=AFinv} we relate $\cA^{inv}$ to the centre
$Z(\covAcp)$ of the covariantized algebra.
\begin{prop}\label{invprops}
  The following relations hold:
  \begin{enumerate}
    \item $a_{\rf^\prime}b=ab=b_{\rf^\prime}a$ for $a\in\cA^{inv}$ and $b\in \cA$.\label{enn1}
    \item $\cA^{inv}\subseteq Z(\covAcp)$.\label{enn2}
    \item $\cA^{inv}$ is a commutative subalgebra of
    $\cA$.\label{enn3}
    \item $\cA^{inv}=\{a\in \cA\,|\,(1\ot X)\cdot a= (\sigma(X) \ot 1)
    \cdot a\,\mbox{ for all $X\in \Uq$}\}$.\label{enn4}
    \item For any $a\in\cA^{inv}$ and $X,Y\in \Uq$ one has
      \begin{align}\label{form}
        a_{(1)}\ot ((X\ot Y)\cdot a_{(2)})=((\sigma(Y)\ot \sigma(X))\cdot a_{(1)})\ot a_{(2)}
      \end{align}
      and in particular, 
      $a_{(1)}\ot \adrs(X)a_{(2)}=(\adrs(\sigma(X))a_{(1)})\ot a_{(2)}$. \label{enn5}
    \item $\cA^{inv}=\{a\in \cA\,|\, a_{(1)} \ot \Psi(a_{(2)})=a_{(2)}\ot
          \sigma^2(\Psi(a_{(1)}))\}$. \label{enn6}
\end{enumerate}
\end{prop}
\begin{proof}
  Property (\ref{enn1}) follows from Lemma
  \ref{compact-aFb} and the fact that $\vep(l_\rp^\pm(a))=\vep(a)$ for all
  $a\in \cA$. Properties (\ref{enn2}) and (\ref{enn3}) follow
  immediately from property (\ref{enn1}).
 
  To verify (\ref{enn4}) note that by definition of the action $\adrs$ the
  set on the right hand side of (\ref{enn4}) is contained in $\cA^{inv}$.
  To obtain the converse inclusion one verifies for any $a\in \cA^{inv}$ the relation
  \begin{align*}
    (1\otimes X)\cdot a=(\sigma(X_{(1)})\otimes 1)(X_{(2)}\otimes X_{(3)})\cdot a
                       =(\sigma(X)\otimes 1)\cdot a.
  \end{align*}  
  This completes the proof of (\ref{enn4}) which also implies property
  (\ref{enn6}) via the non-degenerate Hopf pairing
  $\langle\cdot,\cdot\rangle$ between  $\Uq$ and $\cqg$.

  It remains to prove (\ref{enn5}). Let $a\in\cA$ and $X\in\Uq$. Then
  \[a_{(1)}\ot ((X\ot 1)\cdot a_{(2)})=
    \langle \sigma(a_{(2)}),X\rangle a_{(1)}\ot a_{(3)}=
    ((1\ot \sigma(X))\cdot a_{(1)})\ot a_{(2)}.
  \]
  Hence it suffices to prove \eqref{form} for $X=1$. Let $a\in\cA^{inv}$
  and $Y\in\Uq$, then the desired formula follows by applying the comultiplication
  $\kow$ of $\cA$ to the identity $(1\ot Y)\cdot a=(\sigma(Y)\ot 1)\cdot a$
  from (\ref{enn4}).
\end{proof}
\begin{eg}\label{q-trace-eg}
  For any $V\in Ob(\cC)$ define a quantum trace $\tr_{q,V}\in C^V$ by
  \begin{align*}
    \tr_{q,V}(X):=\tr_V(X\tau(2\rho))
  \end{align*}
  where $\tr_V$ denotes the trace on the linear space $V$ and the
  argument $X\tau(2\rho)$ is considered as an endomorphism of $V$. It
  follows from \eqref{sigma2} that $\tr_{q,V}\in \cqg^{inv}$ where
  $\cqg$ is a left $\Uq$-module via the action $\ad_r^\ast$ defined by
  \eqref{adrs-def}. The Peter-Weyl decomposition \eqref{PeterWeyl} and
  the fact that $C^{V(\lambda)}\cong V(\lambda)^\ast\ot V(\lambda)$ as
  left $\Uq$-modules imply that the quantum traces
  $\{c_\lambda:=\tr_{q,V(\lambda)}\,|\, \lambda\in P^+\}$ form a
  linear basis of $\cqg^{inv}$.
\end{eg}
One may ask if the inclusions in Proposition \ref{invprops}.(\ref{enn2}) are
equalities. This does indeed hold if one assumes $\cA$ to be a domain.
\begin{prop}\label{ZAF=AFinv}
 If $\cA$ is a domain then $\cA^{inv}=Z(\covAcp)$.
\end{prop}
In view of Proposition \ref{AdAFd} and Proposition \ref{invprops}.(\ref{enn2})
the identity $\cA^{inv}=Z(\covAcp)$ is verified by applying the
following lemma to the locally finite, type one, left $\Uq$-module 
algebra $\covAcp$.

\begin{lem}\label{ZBinv}
  Let $\cB$ be a locally finite, type one, left $\Uq$-module algebra and assume
  that $b^2\neq 0$ for all nonzero elements $b\in \cB$. Then $Z(\cB)\subseteq
  \cB^{inv}$. 
\end{lem}
\begin{proof}
  Without loss of generality we may assume that
  $\Uq=U_q(\slfrak_2(\C))$. During this proof, to simplify notation,
  we write $x$, $y$, and $t$ instead of $x_1$, $y_1$, and
  $t_1=\tau(2\alpha_1)$, respectively.  
  As $\cB^{inv}$ is the isotypical component of $\cB$ corresponding to
  the trivial representation of $\Uq$ one has
  \begin{align}\label{trivial}
    \cB^{inv}\cap x\cB=\{0\}.
  \end{align}

  Consider now a central element $a\in Z(\cB)$. As all weight
  components of $a$ are themselves central we may assume that $ta=c\,a$
  for some $c\in \{q^m\,|\,m\in \Z\}$. Choose $n\in \N_0$ minimal such
  that $x^n a\neq 0$ and $x^{n+1}a=0$. We want to show that
  $n=0$. Assume on the contrary that $n\neq 0$. The centrality of $a$ implies
  \begin{align*}
    a(x^n a)=(x^n a)a.
  \end{align*}
  Acting by $x^n$ from the left and using $\kow(x)=x\ot 1 + t\ot x$ one obtains
  \begin{align*}
    (x^n a)(x^n a)=(q^{2n}c)^n(x^n a)(x^n a).
  \end{align*}  
  By assumption the nonzero element $b:=x^n a$ satisfies $b^2\neq 0$ and hence
  $c=q^{-2n}$ since $n\neq 0$. One thus obtains $t x^n a=x^n a$ which, together
  with $x^{n+1}a=0$ and the locally finiteness of $\cB$ implies
  $y(x^n a)=0$. Hence $x^n a\in \cB^{inv}\setminus\{0\}$ which by
  \eqref{trivial} is a contradiction to $n\neq 0$. In the same way one
  shows $ya=0$ and hence $c=1$ since $\cB$ is a locally finite, type
  one $\Uq$-module. This implies $a\in \cB^{inv}$.
\end{proof}
Observe that in the setting of Lemma \ref{ZBinv}, we do not necessarily
have the inclusion $\cB^{inv}\subseteq Z(\cB)$.
\begin{eg}
By Proposition \ref{lr-iso} one has an isomorphism
$\ltil_\rp:\cqg_\rp\rightarrow F_l(U)$ of $\Uq$-module algebras. Hence
Proposition \ref{ZAF=AFinv} for $\cA=\cqg$ implies the well known fact that
$Z(F_l(\Uq))=F_l(\Uq)^{inv}=\ltil_{\rp}(\cqg^{inv})$.  
\end{eg}
\begin{eg}
  Let $V$ be the vector representation of $U_q(\slfrak_n)$ and let
  $\cA(R_{V,V})$ be the corresponding FRT algebra considered in
  subsection \ref{FRT-RE}. It is well known that $\cA(R_{V,V})$ is a
  domain. Hence Proposition \ref{ZAF=AFinv} implies $Z(\cA(R_{V,V})_\rp)=\cA(R_{V,V})^{inv}$ for the
 corresponding reflection equation algebra $\cA(R_{V,V})_\rp$.
\end{eg}
\subsection{The comultiplication as an algebra homomorphism of
  $\cA^{inv}$}\label{comultHom}
The following proposition links the invariant subalgebra $\cA^{inv}$
in a different way with the covariantized algebras $\covAc$ and
$\covAcb$. Let $\mu:= \kow|_{\mathcal{A}^{inv}}$ denote the
restriction of the comultiplication $\kow$ to $\cA^{inv}$. We identify
its range with the product algebra $\covAc\ot \covAcb^\op$.  
\begin{prop}\label{AinvHom}
The map $\mu:=\kow|_{\mathcal{A}^{inv}}$ defines an injective
algebra homomorphism
\[\mu: \mathcal{A}^{inv}\rightarrow \covAc\otimes
\covAcb^\op.
\]
\end{prop}
\begin{proof}
  Assume that $a,b\in \cA^{inv}$. By  Proposition \ref{invprops}.(\ref{enn6})
 we have
  \begin{align}\label{aux1}
    b_{(1)}\ot b_{(2)}\ot \Psi(b_{(3)})=b_{(2)}\ot b_{(3)}\ot \sigma^2(\Psi(b_{(1)}))
  \end{align}
  Using Lemma \ref{compact-aFb} one now calculates
  \begin{align*}
     a_{(1)}{}_\rf b_{(1)}\ot  b_{(2)}{}_{\rfb_{21}} a_{(2)}&\stackrel{(\ref{aFbL1})}{=}
     \big(\adrs[l_\rf^+(\sigma(b_{(1)}))] a_{(1)}\big)b_{(2)}\ot b_{(3)}
     {}_{\rfb_{21}} a_{(2)}\\
    &\stackrel{(\ref{aux1})}{=}
   \big(\adrs[l_\rf^+(\sigma^{-1}(b_{(3)}))] a_{(1)}\big)b_{(1)}\ot b_{(2)}
     {}_{\rfb_{21}} a_{(2)}\\
    &\stackrel{\phantom{(\ref{aux1})}}{=}
   a_{(1)} b_{(1)}\ot b_{(2)} {}_{\rfb_{21}}
     \big(\adrs[l_\rf^+(b_{(3)})]a_{(2)}\big) \\
    &\stackrel{\phantom{(\ref{aFbL1})}}{=}  a_{(1)} b_{(1)}\ot  a_{(2)} b_{(2)}
  \end{align*}
  where we used Proposition \ref{invprops}.(\ref{enn5}) for the third equality
  and relations (\ref{rrelation}) and (\ref{aFbL2}) to obtain the last
  equality. This completes the proof of the proposition.
\end{proof}
\begin{rema}
  The theory of transmutation as introduced by S.~Majid includes the
  construction of a new algebra structure on $\covAcp^{\ot 2}$ which turns the
  comultiplication map $\kow$ of $\cA$ into and algebra homomorphism
  $\kow:\covAcp\rightarrow \covAcp^{\ot 2}$ \cite[10.3.1,
  Proposition 32]{b-KS}. In Majid's terminology
  this turns $\covAcp$ into a so-called braided bialgebra (not to be
  confused with the notion of a quasitriangular bialgebra). The algebra
  homomorphism in Proposition \ref{AinvHom} bears no immediate relation to the
  braided bialgebra structure on $\covAcp$.
\end{rema}
\subsection{Elementary properties of characters}\label{charProp}
To simplify later considerations on characters of the covariantized
algebra $\covAcp$, we first collect some general properties of algebra
homomorphisms from $\covAcp$ to an arbitrary algebra.
\begin{lem}\label{AlgHomLem}
  Let $\cB$ be a $\qfield$-algebra and $f:\cA\rightarrow \cB$ a linear
  map. \\
  {\bf (i)} The following are equivalent:
  \begin{enumerate}
    \item The map $f:\covAcp\rightarrow \cB$ is an algebra homomorphism.
    \item For all $a,b\in \cA$ one has
      \begin{align}
        f(ab)=\rp(\sigma(a_{(1)}),b_{(1)})\rp(a_{(3)},b_{(2)})\, f(a_{(2)})
         f(b_{(3)}).\label{fhom1}
      \end{align} 
    \item For all $a,b\in \cA$ one has
      \begin{align}
        f(ab)&= \rp(b_{(2)},a_{(1)}) \rp(\sigma(b_{(3)}),a_{(3)})\,
           f(b_{(1)}) f(a_{(2)}).\label{fhom2}
     \end{align} 
  \end{enumerate}
 {\bf (ii)} If $f:\cA_\rp\rightarrow \cB$ is an algebra homomorphism then the relations
     \begin{align}
        f(a)f(b)&= \langle
                \sigma(a_{(1)})a_{(3)},l_\rp^-(b_{(3)})\sigma^{-1}(l_\rp^+(b_{(1)}))\rangle
                \,f(b_{(2)})f(a_{(2)}) \label{fafb1}\\
                &= \langle
                \sigma(b_{(1)})b_{(3)},l_\rp^+(a_{(1)})\sigma(l_\rp^-(a_{(3)}))\rangle
                \,f(b_{(2)})f(a_{(2)}) \label{fafb2}
     \end{align}
     hold for all $a,b\in \cA$.
\end{lem}
\begin{proof}
  {\bf (i)} Using properties (\ref{r-mult}), (\ref{r-mult2}), and (\ref{r-vep}) of the
  universal $r$-form $\rp$ one verifies that condition (2) is
  equivalent to the relation
  \begin{align*}
    \rp(a_{(1)},b_{(2)})\rp(a_{(3)},\sigma(b_{(1)}))f(a_{(2)}b_{(3)})=f(a)f(b)
  \end{align*}
  for all $a,b\in \cA$. In view of (\ref{rrrel1}) this relation is
  equivalent to $f:\covAcp\rightarrow \cB$ being an algebra
  homomorphism. The equivalence of properties (1) and (3) follows in
  the same way from (\ref{rrrel2}).

  \noindent {\bf (ii)} Note that \eqref{rrrel1} can be rewritten as
  \[a_{\rf^\prime}b=\langle\sigma(a_{(1)})a_{(3)},
    \sigma^{-1}(l_{\rf^\prime}^+(b_{(1)}))\rangle a_{(2)}b_{(2)}
  \]
  and \eqref{fhom2} as
  \[f(ab)= \langle
  \sigma(a_{(1)})a_{(3)},l_{\rf^\prime}^-(b_{(2)}) 
    \rangle f(b_{(1)})f(a_{(2)}).
  \]
  Combining these expressions we get for $a,b\in\cA$ the relation
  \begin{align*}
    f(a)f(b)=f(a_{\rf^\prime}b)= \langle\sigma(a_{(1)})a_{(3)},
    l_{\rf^\prime}^-(b_{(3)})\sigma^{-1}(l_{\rf^\prime}^+(b_{(1)}))
    \rangle f(b_{(2)})f(a_{(2)}),
  \end{align*}
  which coincides with \eqref{fafb1}. Formula \eqref{fafb2} is verified in a
  similar manner, now using \eqref{rrrel2} and \eqref{fhom1}
  instead of \eqref{rrrel1} and \eqref{fhom2}.
\end{proof}
Recall that we denote the set of characters of a unital
$\qfield$-algebra $\cB$ by $\cB^\wedge$. 
For any two linear functionals $f,g$ on a coalgebra $\cD$ we denote
their convolution product by $f\ast g$. Recall that the convolution
product turns the linear dual space $\cD^\ast$ into a unital algebra.
In the following lemma we view characters of $\cA_\rp$ as linear
functionals on the coalgebra $\cA$.  
\begin{lem}\label{characterLem}
{\bf (i)} Suppose $f\in\cA^\ast$ is convolution invertible with convolution inverse $\fbar\in\cA^\ast$.
Then $f\in \covAc^\wedge$ if and only if $\fbar\in \covAcb^\wedge$.

\noindent{\bf (ii)} If $\mathcal{A}$ is a Hopf algebra then any $f\in
\covAcp^\wedge$ is convolution invertible with convolution inverse $\fbar$
satisfying
      \begin{equation*}
        \begin{split}
          \fbar(a)&= 
\rf^\prime(\sigma(a_{(1)}),a_{(5)})f(\sigma(a_{(2)}))
\rf^\prime(\sigma^2(a_{(3)}),a_{(4)})\\
               &=
\rf^\prime(\sigma^2(a_{(2)}),a_{(3)})f(\sigma^{-1}(a_{(4)}))
\rf^\prime(\sigma(a_{(1)}),a_{(5)})
                    \quad \mbox{for all $a\in \cA$.}
        \end{split}
       \end{equation*}
\end{lem}
\begin{proof}
{\bf (i)} Let $f\in \cA^\ast$ be convolution invertible with
convolution inverse $\fbar$. By Lemma \ref{AlgHomLem}.(i) the functional $f$ is a character
of $\cA_\rf$ if and only if $f(1)=1$ and
\begin{align}\label{f-condition}
  f\circ m=\rfb\ast(f\ot \vep)\ast \rf\ast(\vep\ot f)
\end{align}
where $m$ denotes the multiplication of $\cA$ and $\ast$ is the
convolution product of $(\cA\ot \cA)^\ast$. Note that $f$ satisfies
(\ref{f-condition})  if an only if $\fbar$ satisfies the relation
\begin{align*}
 \fbar\circ m
  =(\vep\ot \fbar)\ast\rfb\ast(\fbar\ot \vep)\ast\rf.
\end{align*}
The latter relation is equivalent to
\begin{align}\label{characterFb}
  \fbar(ab)=\fbar(b_{(1)})\rfb_{21}(b_{(2)},a_{(1)})\fbar(a_{(2)})\rfb_{21}(\sigma(b_{(3)}),a_{(3)})
\end{align}
for all $a,b\in \cA$, which again by Lemma \ref{AlgHomLem}.(i) holds
if and only if $\fbar:\covAcb\rightarrow \qfield$ is an algebra
homomorphism. Moreover, $f(1)=1$ if and only if $\fbar(1)=1$.

{\bf (ii)} If $\cA$ is a Hopf algebra and $f\in \covAcp^\wedge$ then we can define two functionals $\fbar_r$ and
$\fbar_l$ on $\cA$ by
\begin{align*}
  \fbar_r(a)&=\rp(\sigma^2(a_{(2)}),a_{(3)})f(\sigma^{-1}(a_{(4)}))\rp(\sigma(a_{(1)}),a_{(5)}),\\
  \fbar_l(a)&= \rp(\sigma(a_{(1)}),a_{(5)})f(\sigma(a_{(2)}))\rp(\sigma^2(a_{(3)}),a_{(4)})
\end{align*}
for all $a\in \cA$. Applying (\ref{fhom2}) to
$f(\sigma^{-1}(a_{(2)})a_{(1)})$ one obtains that $\fbar_r$ is a right
convolution inverse of $f$. Similarly, applying (\ref{fhom1}) to
$f(\sigma(a_{(1)})a_{(2)})$ one obtains that $\fbar_l$ is a left
convolution inverse of $f$. Hence $f$ is convolution invertible with
inverse $\fbar=\fbar_r=\fbar_l$.
\end{proof}
\subsection{Characters and universal cylinder forms}\label{CylForms}

As an application of Lemma \ref{AlgHomLem}.(i) we now explain that
in our setting a universal cylinder form is the same as a character of a
covariantized bialgebra.
Recall the notion of a universal cylinder form defined in subsection \ref{CylFormDef}.
Restrict again to the case where $\field=\qfield$ and
$(A,\rf_A)$ is of the form considered in Subsection \ref{framework}.
To obtain the desired identification we need the linear twist
functional $u:\cqg\rightarrow \qfield$ defined by 
\begin{align*}
  u(c):=q^{-(\mu+2\rho,\mu)} \vep(c)
\end{align*}
for all $c\in C^{V(\mu)}$, $\mu\in P^+$. Note that by construction 
\begin{align}\label{u-homom}
  c_{(1)}u(c_{(2)})=u(c_{(1)})c_{(2)}
\end{align}
for all $c\in \cqg$. Moreover, it is a well known analogue of
\cite[8.4.3, Proposition 22]{b-KS} that in our setting the relation
\begin{align}\label{rSquared}
  \rf(d_{(1)},c_{(1)}) \rf(c_{(2)},d_{(2)})  u(c_{(3)}d_{(3)})=u(c)u(d)
\end{align}
holds for all $c,d\in \cqg$. Following the general conventions of
Subsection \ref{framework} we consider $u$ as a functional on $\cA$
and again suppress the
homomorphism $\Psi$ in our notation. We are now able to formulate the desired identification.
Let $\covAc^\times$ denote the set of convolution invertible
characters on $\covAc$. Note that $(\cA^\op,\rf_{21})$ is a
coquasitriangular bialgebra.
\begin{prop}\label{CFA}
  The map
  \begin{align}\label{CFAmap}
    CF(\cA^\op,\rf_{21})\rightarrow \covAc^\times,\qquad f\mapsto
    f\ast u
  \end{align}
  is a bijection.
\end{prop}
\begin{proof} 
  Note first that $u$ is convolution invertible with inverse $\ubar$
  defined by  
  \begin{align*}
    \ubar(c):=q^{(\mu+2\rho,\mu)} \vep(c)\qquad \mbox{for all $c\in C^{V(\mu)}$, $\mu\in
    P^+$.}
  \end{align*}
  Given $f\in  CF(\cA^\op,\rf_{21})$ define $g:=f\ast u$ and note that
  $g$ is convolution invertible because both $f$ and $u$ are. Using
  (\ref{cyleq1}) and (\ref{rSquared}) one calculates
  \begin{align*}
    g(ab)& \stackrel{\phantom{(\ref{u-homom})}}{=} f(a_{(1)}b_{(1)})u(a_{(2)}b_{(2)})\\
         & \stackrel{\phantom{(\ref{u-homom})}}{=}  
              f(b_{(1)})\rf_{21}(a_{(1)},b_{(2)}) f(a_{(2)}) \rf_{21}(b_{(3)},a_{(3)})\\
        &\qquad\qquad\times \rf(\sigma(a_{(4)}),b_{(4)}) \rf(\sigma(b_{(5)}),a_{(5)})
              u(a_{(6)})u(b_{(6)})\\
         &\stackrel{(\ref{u-homom})}{=}f(b_{(1)})u(b_{(2)})\rf(b_{(3)},a_{(1)})
         f(a_{(2)})u(a_{(3)})\rf(\sigma(b_{(4)}),a_{(4)})\\
         &\stackrel{\phantom{(\ref{u-homom})}}{=}  g(b_{(1)})\rf(b_{(2)},a_{(1)})
         g(a_{(2)})\rf(\sigma(b_{(3)}),a_{(3)}).
  \end{align*}
  Hence Lemma \ref{AlgHomLem}.(i) implies $g\in \covAc^\times$.
  One checks analogously to the above calculation that the
  map $\covAc^\times\rightarrow CF(\cA^\op,\rf_{21})$, $g\mapsto g\ast
  \ubar$ is well defined. Hence (\ref{CFAmap}) is indeed a bijection.
\end{proof}

\section{Noumi coideal subalgebras of $\Uq$} \label{Noumi}
The essential ingredient in Noumi's construction of quantum symmetric
pairs (e.g.~\cite{a-Noumi96}, \cite{a-NS95}, \cite{a-Dijk96}, \cite{a-NDS97}, \cite{a-DS99}) is a
solution of the reflection equation. We have seen in
Subsection \ref{FRT-RE} that a solution of the reflection equation is the
same as a character of the reflection equation algebra, which in turn
is obtained from the FRT algebra via transmutation. In this section we formulate a
generalised Noumi type construction of coideal subalgebras of $\Uq$ in
terms of characters of the covariantized algebra $\covAcp$. 
\subsection{$\Uq$-comodule algebra structure on $\covAcp$}
The left locally finite part $F_l(\Uq)$ is a left coideal subalgebra
of $\Uq$. Using the $l$-functionals from Subsection \ref{locfinl} and
Proposition \ref{CalderoProp} the coproduct of elements in $F_l(\Uq)$ can
be written as
\begin{align}\label{ltil-kow}
  \kow(\ltil_{\rp}(a))=l_\rp^+(a_{(1)})\sigma( l_\rp^-(a_{(3)}))\ot \ltil_\rp(a_{(2)})
\end{align}
for all $a\in \cqg$ \cite[10.1.3, Proposition 11]{b-KS}. The left
$\Uq$-comodule structure $\kow|_{F_l(\Uq)}$ of $F_l(\Uq)$ can 
be lifted to a left $\Uq$-comodule structure on the covariantized
algebra $\covAcp$. The map $\ltil_\rp:\cA_\rp\rightarrow F_l(\Uq)$ will turn
out to be $\Uq$-comodule algebra homomorphisms. 
Recall our conventions concerning $l$-functionals on $\cA$ from
Subsection \ref{framework} and define a linear map
\begin{align}
  d_\rp&:\covAcp \rightarrow \Uq\ot \cA_\rp, &
  d_\rp(a)&:=l_\rp^+(a_{(1)})\sigma(l_\rp^-(a_{(3)}))\ot a_{(2)}.\label{dr-def}
\end{align}
It is checked that $d_\rp$ defines a $\Uq$-comodule
structure on $\cA$. According to the following lemma, this structure is compatible with the algebra structure of
$\cA_\rp$.
\begin{lem}\label{comod-alg}
  The left coaction $d_\rp$ turns $\covAcp$ into a left $\Uq$-comodule
  algebra. The map $\ltil_\rp:\covAcp\rightarrow F_l(\Uq)$ is a
  homomorphism of left $\Uq$-comodule algebras. 
\end{lem}
\begin{proof}
  Note that the second statement follows immediately from the first statement
  and from comparison of \eqref{ltil-kow} with \eqref{dr-def}.
  We now verify that $(\covAcp,d_\rp)$ is a comodule algebra. For any $a,b\in \cA$ one calculates
  \begin{align*} 
    &d_\rp(a_\rp b)=\\
    &\stackrel{(\ref{rrrel1})}{=}l_\rp^+(a_{(2)}b_{(3)})\sigma(l_\rp^-(a_{(4)}b_{(5)}))
    \rp(a_{(5)},\sigma(b_{(1)}))\rp(a_{(1)},b_{(2)})\ot
    a_{(3)}b_{(4)}\\ 
&\stackrel{(\ref{lpmmult})}{=}l_\rp^+(b_{(3)})l_\rp^+(a_{(2)})\sigma(l_\rp^-(a_{(4)}))\sigma(l_\rp^-(b_{(5)}))
    \rp(a_{(5)},\sigma(b_{(1)}))\rp(a_{(1)},b_{(2)})\ot
    a_{(3)}b_{(4)}\\ 
&\stackrel{(\ref{lpmrtt})}{=}l_\rp^+(a_{(1)})l_\rp^+(b_{(2)})\sigma(l_\rp^-(a_{(4)}))\sigma(l_\rp^-(b_{(5)}))
    \rp(a_{(5)},\sigma(b_{(1)}))\rp(a_{(2)},b_{(3)})\ot
    a_{(3)}b_{(4)}\\ 
&\stackrel{(\ref{lmprtt})}{=}l_\rp^+(a_{(1)})\sigma(l_\rp^-(a_{(5)}))l_\rp^+(b_{(1)})\sigma(l_\rp^-(b_{(5)}))
    \rp(a_{(4)},\sigma(b_{(2)}))\rp(a_{(2)},b_{(3)})\ot
    a_{(3)}b_{(4)}\\ 
&\stackrel{(\ref{rrrel1})}{=}l_\rp^+(a_{(1)})\sigma(l_\rp^-(a_{(3)}))l_\rp^+(b_{(1)})\sigma(l_\rp^-(b_{(3)}))
    \ot a_{(2)}{}_\rp b_{(2)}\\ 
&\stackrel{\phantom{(\ref{aFb1})}}{=} d_\rp(a)d_\rp(b)
  \end{align*}
which completes the proof of the lemma.
\end{proof}

\subsection{Coideal subalgebras of $\Uq$}
Consider $f\in \covAcp^\wedge$. Analogously to the construction of $\cqof$ in subsection \ref{q-conjugacy}
one can use the coaction $d_\rp$ defined by (\ref{dr-def}) to obtain an
algebra homomorphism
\begin{align*}
    d_f:\covAcp\rightarrow \Uq,\qquad d_f(a):=(\id\ot f)d_\rp(a).
  \end{align*}
We define 
\begin{align*} 
  B_f:=d_f(\cA_\rp).
\end{align*}
Note that $d_f$ and hence $B_f$ depend on our choice of universal
$r$-form $\rp=\rf$ or $\rp=\rfb_{21}$. This dependence is only
implicit in our notation via the choice of the character $f$.
\begin{lem}\label{NoumiProps}
  For any character $f\in\covAcp^\wedge$ the following hold:
  \begin{enumerate}
    \item $B_f$ is a left coideal subalgebra of $\Uq$. \label{g1}
    \item $B_f=\{l_\rp^+(a_{(1)})f(a_{(2)})\sigma(l_\rp^-(a_{(3)}))\,|\,a\in
    \cA\}$.\label{g2}
    \item $d_f(\cA^{inv})\subseteq Z(B_f)$.\label{g3}
  \end{enumerate}
\end{lem}
\begin{proof}
  The first two claims hold by definition of $d_f$ and Lemma \ref{comod-alg}. The last claim
  follows from the relation $\cA^{inv}\subseteq Z(\covAcp)$ in Proposition \ref{invprops}.(\ref{enn2}).
\end{proof}  
We call $B_f$ the \textit{Noumi coideal subalgebra} of $\Uq$
corresponding to the character $f$ of $\covAcp$. 

Property (\ref{g2}) of Lemma \ref {NoumiProps} implies in particular
that $B_f=\qfield 1$ if $\Psi(\cA)=\qfield 1$. We give a slight generalisation
of this fact which will be useful in the proof of Proposition
\ref{inLocFinProp}.
To this end note that if $\gfrak=\gfrak_1\oplus \gfrak_2$ is a
decomposition into semisimple Lie subalgebras then there is a
tensor product decomposition of Hopf algebras $\Uq\cong U_1\ot U_2$
where $U_i\cong U_q(\gfrak_i)$ for $i=1,2$. Moreover, in this case
$\cqg\cong\qfield_q[G_1]\ot \qfield_q[G_2]$ for the semisimple, simply connected
affine algebraic groups $G_1$ and $G_2$ corresponding to $\gfrak_1$
and $\gfrak_2$, respectively. Finally, if $\rf_1$ and $\rf_2$ denote
the universal $r$-forms of $\qfield_q[G_1]$ and $\qfield_q[G_2]$,
respectively, then $\rf=\rf_1\ot \rf_2$. 
\begin{lem}\label{semiLem}
  Assume that $\gfrak=\gfrak_1\oplus \gfrak_2$ and $U\cong U_1\ot U_2$ as
  above. If $\Psi(\cA)$ is trivial as a left (or right) $U_1$-module
  with respect to the regular $\Uq$-action on $\cqg$ then $B_f\subseteq
  1\ot U_2 \subseteq \Uq$. 
\end{lem}
\begin{proof}
  If $a\in \cqg$ is invariant under the action of $U_1$ then $a\in
  1\ot \qfield_q[G_2]$. Hence $l_\rp^+(a)\in 1\ot U_2$ and $l_\rp^-(a)\in 1\ot
  U_2$. The claim now follows from Lemma \ref {NoumiProps}.(\ref{g2}).
\end{proof}

\subsection{Duality}
Let $f$ be a character of $\covAcp$. We collect some elementary
relations between the Noumi coideal subalgebra $B_f$ and the quantum
adjoint orbit $\cqof$ defined in subsection \ref{q-conjugacy}.
Note that by definition $\vep\circ \delta_f=\vep\circ d_f=f$.
  \begin{lem}\label{dualLem}
    For any $f\in \covAcp^\wedge$ the following hold:
    \begin{enumerate}
      \item $a(X)=\vep(a)\vep(X)$ for all $X\in B_f$ and $a\in \cqof$.\label{f1}
      \item $f(\adrs(X)a)=\vep(X)f(a)$ for all $X\in B_f$ and $a\in
      \covAcp$. \label{f2}
    \end{enumerate}
  \end{lem}
\begin{proof}
  Let 
  \begin{align}
    X&=d_f(b)=l_\rp^+(b_{(1)})f(b_{(2)})\sigma(l_\rp^-(b_{(3)}))\label{X-def}\\
    a&=\delta_f(c)=\sigma(\Psi(c_{(1)}))f(c_{(2)})\Psi(c_{(3)})\nonumber
  \end{align} 
  for some $b,c\in \covAcp$. Then formula (\ref{fafb2}) implies
  \begin{align*}
    a(X)=\langle \sigma(c_{(1)})c_{(3)},
    l_\rp^+(b_{(1)})\sigma(l_\rp^-(b_{(3)})\rangle f(c_{(2)})f(b_{(2)})=f(b)f(c)=\vep(a)\vep(X)
  \end{align*} 
  and hence (\ref{f1}) holds. Moreover, for $X$ as above and any $a\in
  \covAcp$ one calculates
  \begin{align*}
    f(\adrs(X)a)&\stackrel{(\ref{ad-def})}{=}\langle \sigma(a_{(1)})
     a_{(3)},X \rangle f(a_{(2)})\\
     &\stackrel{(\ref{X-def})}{=}\langle
   \sigma(a_{(1)})a_{(3)}, l_\rp^+(b_{(1)})\sigma(l_\rp^-(b_{(3)}))\rangle
     f(a_{(2)})f(b_{(2)})\\
    &\stackrel{(\ref{fafb2})}{=}f(b)f(a)\\
      &\stackrel{\phantom{(\ref{fafb1})}}{=}\vep(X)f(a)
  \end{align*}
  which proves (\ref{f2}). 
\end{proof}
  For any subset $B\subseteq \Uq$ define
  \begin{align*}
    \cqg^B:=\{a\in\cqg\,|\,\langle a_{(1)},X\rangle\, a_{(2)}=\vep(X)a
    \,\mbox{ for all $X\in B$}\}. 
  \end{align*}
  Note that by construction $\cqg^B$ is a right coideal of
  $\cqg$. Moreover, if $B$ is a (left or right) coideal of $\Uq$ then
  $\cqg^B$ is a right coideal subalgebra of $\cqg$.
  \begin{cor}
    For any $f\in \covAcp^\wedge$ one has $\cqof\subseteq \cqg^{B_f}$.
  \end{cor}
  \begin{proof}
    For $a\in \cqof$ and $X\in B_f$ Lemma \ref{dualLem}.(\ref{f1})
    implies
    \begin{align*}
      \langle a_{(1)},X\rangle \,a_{(2)}=\vep(a_{(1)})\vep(X)a_{(2)}=\vep(X)a
    \end{align*}
    because $\cqof$ is a right coideal subalgebra of $\cqg$.
  \end{proof}
\begin{rema}
  It would be desirable to have a general condition for characters
  $f\in \covAcp^\wedge$ which implies the equality
  $\cqof=\cqg^{B_f}$. For quantum Grassmann manifolds this equality
  holds by \cite[Proposition 2.4]{a-NDS97} which was given without proof.
\end{rema}
\subsection{The centre of Noumi coideal subalgebras}\label{NoumiCentre}
We give a second construction of the central subalgebra
$d_f(\cA^{inv})$ of  $B_f$ obtained in Lemma
\ref{NoumiProps}.(\ref{g3}). Recall from Proposition \ref{AinvHom}
that the coproduct can be  
used to define an algebra homomorphism $\mu: \mathcal{A}^{inv}\rightarrow \covAc\otimes
\covAcb^\op$. Recall moreover that
$\sigma(F_l(U))=F_r(\Uq)$ and that by Proposition \ref{lr-iso} the map
$\sigma\circ\ltil_{\rfb_{21}}:\covAcb^\op\rightarrow F_r(\Uq)$ is an algebra
homomorphism. For $f\in \cA^\wedge_\rf$ we now define an algebra homomorphism
$\Phi_f:\cA^{inv}\rightarrow F_r(\Uq)$ as the composition
\begin{align*}
  \Phi_f: \cA^{inv}\stackrel{\mu}{\longrightarrow}\covAc\ot \covAcb^\op
  \stackrel{f\ot \mathrm{id}}{\longrightarrow} \covAcb^\op\stackrel{\sigma\circ\ltil_{\rfb_{21}}}{\longrightarrow}
  F_r(\Uq),
\end{align*}
or more explicitly
\begin{align*}
  \Phi_f(a):= f(a_{(1)})\,\sigma(\ltil_{\rfb_{21}}(a_{(2)}))
\end{align*}
for all $a\in \cA^{inv}$.
\begin{lem}\label{df=Phif}
  {\bf (i)} For all $c\in \cA^{inv}$ and $f\in \covAc^\wedge$ one has $d_f(c)=\Phi_f(c)$.\\
  {\bf (ii)} For all  $f\in \covAc^\wedge$ one has $d_f(\cA^{inv})\subseteq Z(B_f)\cap F_r(\Uq)$.
\end{lem}
\begin{proof}
  Applying the coproduct to the relation in Proposition \ref{invprops}.(\ref{enn6}) one obtains
  \begin{align}\label{c3rel}
     c_{(1)}\ot c_{(2)} \ot \Psi(c_{(3)})= c_{(2)}\ot c_{(3)}\ot \sigma^2(\Psi( c_{(1)}))
  \end{align}
  for all $c\in \cA^{inv}$. By definition of $\Phi_f$ and
  (\ref{ltildef}), (\ref{rrelation}) one now gets
  \begin{align*}
    \Phi_f(c)= f(c_{(1)})\sigma^2(l_\rf^+( c_{(3)}))\sigma(l_\rf^-(c_{(2)}))
    \stackrel{(\ref{c3rel})}{=}l_\rf^+( c_{(1)})
    f(c_{(2)})\sigma(l_\rf^-(c_{(3)}))
    =d_f(c)
  \end{align*}
   for all $c\in \cA^{inv}$ which proves {\bf (i)}. Claim {\bf (ii)}
   follows from {\bf (i)} and the inclusion $\cA^{inv}\subseteq Z(\covAc)$.
\end{proof}

The subalgebra $d_f(\cA^{inv})$ of $Z(B_f)\cap
F_r(\Uq)$ is of particular interest if the character $f$ factors
through $\Psi$ or, more explicitly, if $\ker(\Psi)\subseteq \ker(f)$. In
this case, if in addition $\gfrak$ is simple, we may assume that
$\cA=\qfield_q[G_L]$ for some lattice $L\subset \hfrak^\ast$ such that
$Q\subseteq L\subseteq P$ by Proposition  \ref{cqGLprop}. Throughout
this subsection, for arbitrary semisimple $\gfrak$, we denote by
$L\subset \hfrak^\ast$ a lattice such that $Q\subseteq L\subseteq P$.
\begin{lem}\label{inj-homom}
   If $\cA=\qfield_q[G_L]$ and $f\in \covAc^\wedge$ then the map
   \begin{align*}
     d_f|_{\cA^{inv}}:\cA^{inv}\rightarrow Z(B_f)\cap F_r(\Uq)
   \end{align*}
    is an injective algebra homomorphism.
\end{lem}
\begin{proof}
  Assume $d_f(a)=0$ for some $a\in \cA^{inv}$. This implies
  $f(a_{(1)})\sigma(\ltil_{\rfb_{21}}(a_{(2)}))=0$ by Lemma
  \ref{df=Phif}.(i). By Proposition \ref{CalderoProp} the map $\sigma\circ
  \ltil_{\rfb_{21}}$ is injective and
 hence $f(a_{(1)})a_{(2)}=0$. We now use the fact that $f$ is
  convolution invertible from Lemma \ref{characterLem}.(ii) to obtain
  \begin{align*}
    a=f(a_{(1)})\fbar(a_{(2)})a_{(3)}=(\fbar\ot \id)\kow(f(a_{(1)})a_{(2)})=0.
  \end{align*}
  Hence $d_f|_{\cA^{inv}}$ is indeed injective. The fact that
  $d_f|_{\cA^{inv}}$ is an algebra homomorphism follows from
  Proposition \ref{invprops}.(1),(2) and from the fact that $d_\rf$ is
  an algebra homomorphism by Lemma \ref{comod-alg}.
\end{proof}
Let $\rep(\gfrak)$ denote the representation ring of $\gfrak$,
i.e.~the $\C$-algebra with basis $\{r_\lambda\}_{\lambda\in
P^+}$ and product
\begin{align*}
  r_\lambda r_\mu =\sum_{\nu\in P^+}m^\nu_{\lambda,\mu} r_\nu, \qquad
  \mbox{where  }\quad
  m_{\lambda,\mu}^\nu:=\dim\big(\Hom_\Uq(V(\nu),V(\lambda)\ot V(\mu))\big).
\end{align*}
For any lattice $L\subset \hfrak^\ast$ such that $Q\subseteq L\subseteq P$
let $\rep(\gfrak)_L$ denote the subalgebra of $\rep(\gfrak)$ with
basis $\{r_\lambda\,|\,\lambda\in
L\cap P^+\}$. It was proved for instance in \cite[8.6]{a-JoLet1} that
there exists a basis $\{z_\lambda\}_{\lambda\in P^+}$ of the centre
$Z(\Uq)$ with $z_\lambda\in (\adr \Uq)\tau(2\lambda)$ such that the
map
\begin{align*}
  Z(\Uq)\rightarrow \rep(\gfrak), \qquad z_\lambda\mapsto r_\lambda
\end{align*}
defines an isomorphism of algebras (cp.~also \cite{a-Baumann98}). By
Propositions \ref{CalderoProp}, \ref{lr-iso}, and \ref{ZAF=AFinv} the elements
\begin{align} \label{c'lambda-def}
  c_\lambda':=\ltil_{\rfb_{21}}^{-1}(\sigma^{-1}(z_\lambda))\in
  C^{V(-w_0\lambda)}\cap \cqg^{inv}
\end{align} 
also satisfy $c'_\lambda c'_\mu =\sum_{\nu\in P^+}m^\nu_{\lambda,\mu}
c'_\nu$ and hence yield a realisation of $\rep(\gfrak)$ inside
$\cqg^{inv}$. Note that $c'_\lambda$ is $(\adrs \Uq)$-invariant and
hence coincides with the quantum trace 
$c_{-w_0\lambda}$ defined in Example \ref{q-trace-eg} up to a scalar
factor. The following theorem is now 
an immediate consequence of the fact that for $f\in \cqg^\wedge_\rf$
the map $d_f|_{\cqg^{inv}}$ is an injective algebra homomorphism by
Lemma \ref{inj-homom}. 

\begin{thm}\label{RepThm}
  If $\cA=\qfield_q[G_L]$ and $f\in \covAc^\wedge$ then the map
  \begin{align*}
    \rep(\gfrak)_L\rightarrow Z(B_f)\cap F_r(\Uq), \qquad r_\lambda\mapsto d_f(c'_\lambda)
  \end{align*}
  is an injective homomorphism of algebras.
\end{thm}
\begin{rema}
  Note that by Proposition \ref{CalderoProp} and Lemma \ref{df=Phif} one has for any
  $\lambda\in L\cap P^+$ the relation
  \begin{align}\label{inTau2lambda}
    d_f(c'_\lambda)=\Phi_f(c'_\lambda)\in \sigma\left( (\adl
    \Uq)\tau(-2\lambda)\right)=(\adr \Uq)\tau(2\lambda).
  \end{align}
  Theorem \ref{RepThm} hence states that the centre of the Noumi
  algebra $B_f$ contains a canonical
  subalgebra $d_f(\cA^{inv})$ which is naturally (with respect to the grading coming
  from $F_r(\Uq)$) isomorphic to the representation ring $\rep(\gfrak)_L$. 
\end{rema}

\subsection{Local finiteness}\label{locFiness}
In view of Lemma  \ref{df=Phif}.(ii), Lemma \ref{inj-homom}, and
Theorem \ref{RepThm} it is natural to ask if $Z(B_f)$
is contained in $F_r(U)$ for any Noumi coideal subalgebra $B_f$. We
attack this question in this subsection using a convenient criterion
to determine whether an element in $\Uq$ is contained in $F_r(\Uq)$.
We first recall the following preparatory lemma, which is valid
for an arbitrary Hopf algebra $H$.
\begin{lem}\label{centreLem}
  Let $B\subseteq H$ be a left coideal subalgebra and $C\subseteq H$ a
  right coideal subalgebra of a Hopf algebra $H$. Then
  \begin{align}
    Z(B)=\{b\in B\,|\, (\adr x)b=\vep(x)b\, \mbox{ for all $x\in
    B$}\},\label{ZB}\\
    Z(C)=\{c\in C\,|\, (\adl x)c=\vep(x)c\, \mbox{ for all $x\in
    C$}\}.\label{ZC}
  \end{align}
where $(\adl x)h=x_{(1)}h\sigma(x_{(2)})$ and
$(\adr x)h=\sigma(x_{(1)})hx_{(2)}$ for $x,h\in H$.
\end{lem}
\begin{proof}
  This result is proved in complete analogy to \cite[Lemma
  1.3.3]{b-Joseph}. To see that the right hand sides of (\ref{ZB}) and
  (\ref{ZC}) are contained in $Z(B)$ and $Z(C)$, respectively, one uses
\begin{align*}
  hx=x_{(1)}((\adr x_{(2)})h),\qquad xh=((\adl x_{(1)})h)x_{(2)}
\end{align*}
for any $x,h\in H$. 
\end{proof}
\begin{prop}\label{Ci-prop}
   For $i=1,\dots,r$ let $C_i\in \Uq$  be elements such that
   \begin{align}\label{Ci-form}
     C_i\in x_i\tau(\mu_i)+\bigoplus_{\alpha\le 0} \Uq_{\alpha}
   \end{align} 
   for some $\mu_i\in P$.
   If  $u\in \Uq$ satisfies $(\adl C_i) u=\vep(C_i)u$ for all  $i=1,\dots,r$  
   then $u\in F_l(\Uq)$. Similarly, if $u\in \Uq$
   satisfies  $(\adr C_i) u=\vep(C_i)u$  for all  $i=1,\dots,r$  
   then $u\in F_r(\Uq)$.
\end{prop}
\begin{proof}
  Note that if $x\in \Uq$ satisfies $\dim((\adl \uqbp) x)<\infty$ then $x\in
  F_l(\Uq)$.  The proof of this fact is nicely written up as the first step of the proof of
  \cite[Lemma 3.1.1]{a-Fauqu98}. 
  We now use this fact to prove the first statement of the
  proposition.  Decompose the element $u=\sum_{\gamma\in
  Q} u_\gamma$ where $u_\gamma\in \Uq_\gamma$. It suffices to show by
  induction on $\gamma$ that $\dim((\adl \uqbp) u_\gamma)<\infty$. Fix
  $\beta\in Q$ and assume that  $\dim((\adl \uqbp)
  u_\gamma)<\infty$ for all $\gamma>\beta$. The relation  $(\adl C_i) u=\vep(C_i)u$ and the
  special form (\ref{Ci-form}) of the elements $C_i$ imply
  \begin{align*}
    \vep(C_i)u_{\beta+\alpha_i}\in q^{(\mu_i,\beta)}(\adl x_i)u_\beta
    + \sum_{\gamma\geq\beta+\alpha_i}(\adl U_{\beta-\gamma+\alpha_i})u_\gamma
  \end{align*}
  and hence
  \begin{align*}
    (\adl x_i)u_\beta\in \sum_{\gamma>\beta}(\adl \Uq)u_\gamma
  \end{align*}
  for $i=1,\dots,r$. By induction hypothesis the right hand side of the above expression
  is contained in $F_l(\Uq)$ and hence $\dim((\adl \uqbp)
  u_\beta)<\infty$. This completes the proof of the first
  statement. Using the relations (\ref{ad-leftright}) one
  immediately obtains the second claim of the proposition.
\end{proof}
\begin{rema}
  The second half of the above proof resembles an argument
  given in the proof of \cite[Lemma  4.4]{a-Letzter97}. Following
  G.~Letzter's more general setting one can even show that any element $u\in
  \Uq$ such that $\mathrm{span}\{(\adl C_i^m) u\,|\,m\in \N_0\}$ is
  finite-dimensional for all  $i=1,\dots,r$, belongs to $ F_l(\Uq)$.
\end{rema}
Lemma \ref{centreLem} and Proposition \ref{Ci-prop} now imply
the following local finiteness result for Noumi coideal subalgebras
in $\Uq$.
\begin{prop}\label{inLocFinProp}
  Let $f\in \covAc^\wedge$ be convolution invertible. Then $Z(B_f)\subset F_r(\Uq)$.
\end{prop}
\begin{proof}
  By Lemma \ref{centreLem} any $b\in Z(B_f)$ satisfies
  $\adr((\beta-\vep(\beta))X)b=0$ for all $\beta\in B_f$ and $X\in
  \Uq$. In particular for any $a\in \cA$ one obtains
  \begin{align*}
    \adr\big([d_f(a_{(1)})-f(a_{(1)})]\,l_\rf^-(a_{(2)})\fbar(a_{(3)})\big)b=0.
  \end{align*}
  This relation may be rewritten as
  \begin{align}\label{ainv}
    \adr\big(l_\rf^+(a)-f(a_{(1)})l_\rf^-(a_{(2)})\fbar(a_{(3)})\big)b=0
  \end{align}
  for all $a\in \cA$. Take any generator $x_j\in \Uq$. As above Lemma
  \ref{semiLem}, we can decompose the Hopf algebra $\Uq$ in the form
  $\Uq\cong U_1\ot U_2$ where $x_j\in U_1$ and and $U_i\cong U(\gfrak_i)$,
  $i=1,2$, with simple $\gfrak_1$ and semisimple or trivial $\gfrak_2$. By Lemma
  \ref{semiLem} we may assume that the left or right regular action of
  $U_1$ on $\Psi(\cA)$ is non-trivial. Hence there exists $\lambda\in \Ppiplus$ such that
  $U_1$ acts non-trivially on $V(\lambda)$ and $C^{V(\lambda)}\subset \Psi(\cA)$.
  It follows from (\ref{lexplicit}) that there
  exist weight vectors $v\in V(\lambda)$ and $f\in V(\lambda)^\ast$ such that
  $l_\rf^+(c^\lambda_{f,v})=x_j\tau(-\wght(v))$. Choose $a\in \cA$ such
  that $\Psi(a)=c^\lambda_{f,v}$ and consider the element
  \begin{align*}
    C_j:=l_\rf^+(a)-f(a_{(1)})l_\rf^-(a_{(2)})\fbar(a_{(3)}).
  \end{align*}
  By construction $C_j$ is of the form (\ref{Ci-form}), and by
  (\ref{ainv}) we have $(\adr C_j)b=\vep(C_j)b$. As this construction
  is possible for any $x_j$ one may now apply Proposition \ref{Ci-prop}
  to obtain $b\in F_r(\Uq)$.
\end{proof}

\section{Constructing characters of $\covAc$}\label{ConstrChar}
In Section \ref{Noumi} we explained the relevance of characters of the
covariantized algebra $\covAc$. They are the main ingredient in the
construction of Noumi coideal subalgebras and quantum adjoint orbits.
As explained in Section \ref{CylForms} universal cylinder forms
coincide with characters of covariantized algebras. The case when
$\cA=\cqg$ is of particular interest because it leads to realisations
of $\rep(\gfrak)$ inside $\Uq$. We now address the immediate question 
of how to obtain such characters. 

\subsection{Solutions of the reflection equation from central
  elements}\label{solvingRE}
  Consider an element $C\in F_r(\Uq)$. It follows from the direct sum
  decomposition (\ref{Fr-decomp}) that there exists a finite subset
  $P_C^+$ of $P^+$ such that 
  $C\in \bigoplus _{\mu\in P^+_C}(\adr \Uq)\tau(-2w_0\mu)$. Define
  $c_C:=\sum_{\mu\in P^+_C}c_\mu\in \cqg^{inv}$ as sum of the quantum traces
  $c_\mu$. Note that by definition of quantum traces in Example
  \ref{q-trace-eg} and by Proposition \ref{CalderoProp} there exists a uniquely determined linear
  functional $f_C:\cqg\rightarrow \qfield$ such that
  \begin{align}\label{fC-def}
     C=f_C(c_{C(1)})\sigma\big(\ltil_{\rfb_{21}}(c_{C(2)}) \big)
  \end{align}
  and $f_C(C^{V(\mu)})=0$ for all $\mu\notin P^+_C$.
  Note, moreover, that $f_C$ depends on $C$ only and not on the choice of $P^+_C$.
  We use the functional $f_C$ to reformulate and generalise the observation made in
  \cite[Corollary 2]{a-Kolb08} that suitable central elements in coideal subalgebras
  of $\Uq$ lead to solutions of the reflection equation.
  \begin{prop}\label{RE-Prop}
    Let $B\subseteq \Uq$ be a left coideal subalgebra and $C\in Z(B)\cap
    F_r(\Uq)$. Then the functional $f_C:\cqg\rightarrow \qfield$ defined
    above satisfies the relation
    \begin{align}
            f_C(b_{(1)})\,\rf(b_{(2)},a_{(1)})& \,
      f_C(a_{(2)})\,\rf(\sigma(b_{(3)}),a_{(3)}) \nonumber\\
      &=  \rf(\sigma(a_{(1)}),b_{(1)})\, f_C(a_{(2)})\,\rf(a_{(3)},b_{(2)})\,
         f_C(b_{(3)})\label{RE-form}
    \end{align}
    for all $a,b\in \cqg$.
  \end{prop}
To make the proof of the above proposition more manageable we separate
the main technical step in the following lemma.
\begin{lem}\label{hilfslem}
  Let $B_r\subseteq \Uq$ be a right coideal subalgebra and $D\in Z(B_r)\cap F_l(\Uq)$.
  Then the relation
  \begin{align*}
     \rf(\sigma(m_{(1)})m_{(3)}, \sigma^{-1}(n_{(2)})) \,m_{(2)} \ot
    n_{(1)} = \rf(\sigma^2(n_{(1)}), \sigma(m_{(1)})m_{(3)})\,
    m_{(2)}\ot n_{(2)}
  \end{align*}
  holds for $m=n= \ltil_{\rfb_{21}}^{-1}(D)\in \cqg$.
\end{lem}
\begin{proof}
   It follows from Lemma \ref{centreLem} and the coideal property $\kow (B_r)\subseteq B_r\ot
   \Uq$ that $(\adl D_{(1)})D \ot D_{(2)}=D\ot D$. In view of relations
   (\ref{ltil-kow}) and (\ref{rrelation}) this can be rewritten as 
   \begin{align*}
     \adl  \big( l_\rf^-(n_{(1)})\sigma(l_\rf^+(n_{(3)}))\big)D\ot
     \ltil_{\rfb_{21}}(n_{(2)})=D\ot \ltil_{\rfb_{21}}(n).
   \end{align*}
  In view of Proposition \ref{CalderoProp} the above relation implies
  \begin{align*}
   \adl   \big(l_\rf^-(n_{(1)})\sigma(l_\rf^+(n_{(4)}))\big)D\ot
     n_{(2)}\ot n_{(3)}=D\ot n_{(1)}\ot n_{(2)}
  \end{align*}
  and hence
  \begin{align*}
     \adl   \big(\sigma(l_\rf^+(n_{(2)}))\big)D\ot
     n_{(1)}= \adl \big(\sigma^{-1}(l_\rf^-(n_{(1)}))\big)D \ot  n_{(2)}.
  \end{align*}
  Inserting $D=\ltil_{\rfb_{21}}(m)$ and using the fact that
  $\ltil_{\rfb_{21}}:\cqg\rightarrow F_l(\Uq)$ is an isomorphism of left $\Uq$-modules one obtains
  \begin{align*}
    \langle \sigma(m_{(1)})m_{(3)}, \sigma(l_\rf^+(n_{(2)})) \rangle
    \,m_{(2)} \ot n_{(1)}=  \langle \sigma(m_{(1)})m_{(3)}, \sigma^{-1}(l_\rf^-(n_{(1)})) \rangle\,
    m_{(2)}\ot n_{(2)}.
  \end{align*}
  By definition of $l_\rf^+$ and $l_\rf^-$ this is equivalent to the the desired formula.
\end{proof}
\begin{proof}[Proof of Proposition \ref{RE-Prop}]
  We apply Lemma \ref{hilfslem} to the right coideal subalgebra
  $B_r:=\sigma^{-1}(B)$ of $\Uq$ and to the element
  $D:=\sigma^{-1}(C)$. To simplify notation we define $f:=f_C$, $e:=c_C$, and
  $c:=c_C$. Note that in the notation of Lemma \ref{hilfslem} one
  has $m=f(c_{(1)})c_{(2)}$ and $n=f(e_{(1)})e_{(2)}$ and hence one
  gets
  \begin{align}
    f(c_{(1)}) f(e_{(1)}) &\rf(\sigma(c_{(4)}),e_{(3)})
    \rf(\sigma^2(c_{(2)}),e_{(4)})\, c_{(3)}\ot e_{(2)}\nonumber\\
    &= f(c_{(1)}) f(e_{(1)})\rf(\sigma^2(e_{(2)}),c_{(4)})
    \rf(\sigma(e_{(3)}),c_{(2)})\, c_{(3)}\ot e_{(4)}.\label{firststep}
  \end{align}
  Using Proposition \ref{invprops}.(\ref{enn6}) we now apply
  \begin{align*}
    c_{(1)} \ot c_{(2)} \ot c_{(3)} \ot c_{(4)} = c_{(2)} \ot
    c_{(3)}\ot c_{(4)}\ot \sigma^2(c_{(1)})
  \end{align*}
  to both sides of (\ref{firststep}), and we apply
   \begin{align*}
    e_{(1)} \ot e_{(2)} \ot e_{(3)} \ot e_{(4)} = e_{(3)}\ot
    e_{(4)}\ot \sigma^2(e_{(1)})\ot \sigma^2(e_{(2)})
  \end{align*}
  to the left hand side of (\ref{firststep}). One obtains
  \begin{align*}
     f(c_{(2)}) f(e_{(3)}) &\rf(\sigma(c_{(1)}),e_{(1)})
    \rf(c_{(3)},e_{(2)})\, c_{(4)}\ot e_{(4)}\\
    &= f(c_{(2)}) f(e_{(1)})\rf(e_{(2)},c_{(1)})
    \rf(\sigma(e_{(3)}),c_{(3)})\, c_{(4)}\ot e_{(4)}.
  \end{align*}
  In view of the special form  of quantum trace $c=e=c_C$ this
  proves (\ref{RE-form}) for all $a,b\in \bigoplus_{\mu\in P^+_C}
  C^{V(\mu)}$ and hence for all $a,b\in \cqg$.
\end{proof}
Proposition \ref{RE-Prop} provides characters of the reflection equation
algebra from Subsection \ref{FRT-RE} via suitable central elements in
any coideal subalgebra of $\Uq$. Recall that for any $V\in Ob(\cC)$
the corresponding FRT algebra $\cA(R_{V,V})$ is generated by the
linear space $V^\ast\ot V$ as an algebra. Let $t^h_v\in \cA(R_{V,V})$
denote the generator corresponding to $h\ot v\in V^\ast\ot V$.  The
following Corollary is a direct consequence of Proposition \ref{FRTRE}
and Proposition \ref{RE-Prop}. 
\begin{cor}\label{RE-character}
  Let $B\subseteq \Uq$ be a left coideal subalgebra and let $C\in
  Z(B)\cap F_r(\Uq)$. For any $V\in Ob(\cC)$ there exists a
  unique character
  $g_{C,V}\in\mathcal{A}(R_{V,V})_{\mathbf{r}}^\wedge$ such that 
  $g_{C,V}(t^h_v)=f_C(c_{h,v})$ for all $h\in V^\ast, v\in V$.  
\end{cor}
To make the Noumi coideal subalgebra corresponding to the character $g_{C,V}$
more explicit, we will use the following auxiliary observation. For
any $\mu\in P^+$ let $p_\mu: F_r(U)\rightarrow (\adr \Uq)\tau(2\mu)$
denote the projection map with respect to the direct sum decomposition
\eqref{Fr-decomp}. 
\begin{lem}
  Let $B\subseteq \Uq$ be a left coideal subalgebra and $\mu\in
  P^+$. Then 
  \begin{align*} 
     p_\mu(B\cap F_r(\Uq))\subseteq B\qquad\mbox{ and }\qquad
     p_\mu(Z(B)\cap F_r(\Uq))\subseteq Z(B).
  \end{align*}
\end{lem}
\begin{proof}
   Note that $p_\mu$ is a homomorphism of right coideals in
   $\Uq$. Together with  the fact that $B$ is a left coideal 
   subalgebra this implies the first inclusion. The second inclusion
   is now an immediate consequence of Lemma \ref{centreLem}.
\end{proof}
By the above lemma any $C\in Z(B)\cap F_r(\Uq)$ can be written as a finite sum 
\begin{align}\label{C-decomp}
  C=\sum_{\mu\in P_C^+} C_\mu\qquad\mbox{with $C_\mu\in (\adr
  \Uq)\tau(-2w_0\mu)\cap Z(B)$.}
\end{align}
The Noumi coideal subalgebra $B_{g_{C,V}}$ corresponding to the character from Corollary \ref{RE-character} is obtained from the central elements $C_\mu$ as follows.
\begin{prop}\label{Noumi-identify}
Let $B\subseteq\Uq$ be a left coideal subalgebra and let
$C\in Z(B)\cap F_r(\Uq)$. For $V\in Ob(\cC)$ let $g_{V,C}\in
\cA(R_{V,V})_\rf^\wedge$ be the character obtained in Corollary
\ref{RE-character}. Let $C=\sum_{\mu\in P_C^+} C_\mu$ be the
decomposition \eqref{C-decomp}. Then the Noumi coideal subalgebra
$B_{g_{C,V}}$ is the left coideal subalgebra of $\Uq$ generated (as a
left coideal subalgebra) by the elements $C_\mu$ for all $\mu \in
P^+_C$ with $\Hom_\Uq(V(\mu),V)\neq \{0\}.$ In particular,
one has $B_{g_{C,V}}\subseteq B$. 
\end{prop}
\begin{proof}
  Let $B_\mu\subseteq \Uq$ denote the left coideal generated by the
  element $C_\mu$ and as before let $f_C$ be the linear functional defined by
  \eqref{fC-def}. The relation 
  $C_\mu=f_C(c_{\mu(1)})\sigma(\ltil_{\rfb_{21}}(c_{\mu(2)}))$
  implies in view of relation \eqref{ltil-kow} that
  \begin{align*}
    \kow(C_\mu)&=f_C(c_{\mu(1)})\,\sigma(\ltil_{\rfb_{21}}(c_{\mu(3)}))\ot
    \sigma^2(l^-_{\rfb_{21}}(c_{\mu(4)}))\sigma(l_{\rfb_{21}}^+(c_{\mu(2)}))
    \\
               &=\sigma^{-1}(\ltil_{\rfb_{21}}(c_{\mu(1)}))\ot
    l^+_{\rf}(c_{\mu(2)})f_C(c_{\mu(3)})\sigma(l_{\rf}^-(c_{\mu(4)})) 
  \end{align*}
  where we used Proposition \ref{invprops}.(\ref{enn6}) and
  \eqref{rrelation} for the last equality.
  Hence we obtain
  \begin{align*}
    B_\mu=\{l^+_\rf(e_{(1)})f_C(e_{(2)})\sigma(l^-_\rf (e_{(3)}))\,|\,
    e\in C^{V(\mu)}\}.
  \end{align*}
  On the  other hand, by Lemma \ref{NoumiProps}.(\ref{g2}) the
  algebra $B_{g_{C,V}}$ is generated as an algebra by the subspace
  \begin{align*}
    \{l^+_\rf(e_{(1)})f_C(e_{(2)})\sigma(l_{\rf}^- (e_{(3)}))\,|\,
    e\in C^{V}\}.
  \end{align*}
  This subspace coincides with the span of all $B_\mu$ for $\mu\in
  P^+_C$ such that $V(\mu)$ occurs as a direct summand in $V$.
\end{proof}


\begin{rema}\label{QSP-remark}
  G.~Letzter's family of quantum symmetric pair coideal subalgebras $B$ of $\Uq$
  \cite{MSRI-Letzter} is a very interesting class of examples to
  which Propositions \ref{RE-Prop} and \ref{Noumi-identify} apply. The
  centre of these left coideal subalgebras was determined in
  \cite{a-KL08}. It follows from  \cite[Footnote to Corollary
  8.3]{a-KL08} that for each of these left coideal subalgebras $B$ of $\Uq$
  there exists a subset $P_{Z(B)}^+\subseteq P^+$ such that
  \begin{align}\label{ZB1}
     \dim\big(Z(B)\cap (\adr \Uq)\tau(-2w_0\mu)
     \big)=\begin{cases}1&\mbox{if $\mu\in P_{Z(B)}^+$,} \\ 0
     & \mbox{else}\end{cases}
  \end{align}
  and
  \begin{align}\label{ZB2}
     Z(B)=\bigoplus_{\mu\in P_{Z(B)}^+} Z(B)\cap (\adr \Uq)\tau(-2w_0\mu).
  \end{align}
  Note that this subset was denoted by $P_{Z(B)}$ in \cite{a-KL08}.
  The set $P_{Z(B)}^+$ is explicitly determined in \cite[Proposition
  9.1]{a-KL08}. In many cases, in particular if $\gfrak$ has no
  diagram automorphisms, one has $P_{Z(B)}^+=P^+$. Moreover,
  $P_{Z(B)}^+$ is invariant under taking
  dual weights. Proposition \ref{RE-Prop} implies that for any quantum
  symmetric pair coideal subalgebra $B$ and any $\mu\in P_{Z(B)}^+$ one
  obtains a solution of the reflection equation (\ref{refleqn}) for $V=V(\mu)$. By
  \cite[Proposition 4]{a-Kolb08} this solution is non-diagonal and
  hence no multiple of the identity. In Subsection \ref{Grassmann} we will
  explicitly discuss the quantum symmetric pair corresponding to the
  Grassmannian manifold $Gr(m,2m)$ of $m$-dimensional subspaces in $\C^{2m}$.
 \end{rema}
\subsection{Characters of $F_l(\Uq_q(\slfrak_n))$}\label{characters}
For the rest of this section we restrict to the case where
$\gfrak=\slfrak_n=\slfrak_n(\C)$ and $V=V(\omega_1)$ is the vector
representation of $\Uq=U_q(\slfrak_n)$. Note that
$r=\rk(\slfrak_n)=n-1$. We choose the root system for $\slfrak_n$ and the simple roots
$\{\alpha_1,\dots,\alpha_r\}$ as in \cite[12.1]{b-Humphreys}. Recall that $V$ has a basis 
$\{v_1,\dots,v_n\}$ such that 
\begin{align}\label{vec-rep}
  x_iv_j=\delta_{i,j-1}v_{j-1},\qquad
  y_iv_j=\delta_{i,j}v_{j+1},\qquad 
  t_iv_j=q^{-\delta_{i+1,j}+\delta_{i,j}}v_j.
\end{align}
As in previous sections we let $\{f_1,\dots,f_n\}$ denote the basis of $V^\ast$
dual to $\{v_1,\dots,v_n\}$. To shorten notation we define
$c_{i,j}:=c_{f_i,v_j}$. 

As in subsection \ref{FRT-RE} let $R:=R_{V,V}=P_{12}\circ \rh_{V,V}$ be
the $R$-matrix corresponding to $V$. If we define an $(n^2\times
n^2)$-matrix $(R^{ij}_{kl})$ by  $R^{ij}_{kl}=\rf(c_{i,k},c_{j,l})$
then one has $R(v_k\ot v_l)=\sum_{i,j=1}^n v_i\ot v_j R^{ij}_{kl}$.
It follows from \cite[3.15]{b-Jantzen96} translated to our conventions
that
\begin{align}\label{R-explicit}
  R^{ij}_{kl}=q^{1/n}\begin{cases} q^{-1}& \mbox{if $i=j=k=l$,}\\
                                   1&\mbox{if $i=k\neq j=l$,}\\
                                   q^{-1}-q&\mbox{if $i=l<j=k$,}\\
                                   0&\mbox{else.}                     
                     \end{cases}
\end{align}
Note that this matrix coincides with the matrix in \cite[8.4
  (60)]{b-KS} up to the overall factor and taking the inverse of the transpose. 

Recall the FRT-algebra $\cA(R)$ considered in Subsection
\ref{FRT-RE}. As in that subsection we denote the generators
$\{t^i_j\}$ of $\cA(R)$ by $\{s^i_j\}$ if we consider them as
generators of the reflection equation algebra $\cA(R)_\rf$.
For any element $\sigma$ in the symmetric group $S_n$ let $l(\sigma)$ be its
length. Recall that the quantum determinant 
\begin{align*}
   {\det}_q:=\sum_{\sigma\in S_n} (-q)^{l(\sigma)}t^1_{\sigma(1)} \dots
    t^n_{\sigma(n)}
\end{align*}
is a central element in $\cA(R)$ such that $\cA(R)/(\det_q-1)\cong
\qfield_q[SL(n)]$ (cp.~\cite[9.2.3, 11.2.3]{b-KS}).  If one considers $\det_q$ as an
element in the covariantized algebra $\cA(R)_\rf$ then the latter relation
implies $\cA(R)_\rf/(\det_q-1)\cong \qfield_q[SL(n)]_\rf$. This proves the
first part of the following lemma. 
\begin{lem}\label{charSLn}
   There is a one-to-one correspondence between the set
   $\qfield_q[SL(n)]_\rf^\wedge$ and the set
   $\{f\in\cA(R)_\rf^\wedge\,|\,f(\det_q)=1\}$.  
If $f\in\cA(R)_\rf^\wedge$ is a character such that
   $f(\det_q)=\beta^n$ for some $\beta\in\qfield\setminus\{0\}$ then the map
$g:\qfield_q[SL(n)]_\rf\rightarrow \qfield$ defined by
$g(c_{i,j}):=\beta^{-1} f(s^i_j)$ 
defines a character of $\qfield_q[SL(n)]_\rf\cong F_l(\Uq_q(\slfrak_n))$.    
\end{lem}
\begin{proof}
  For any character $f\in\cA(R)_\rf^\wedge$ such that
  $f(\det_q)=\beta^{n}$ one defines a character $g
  \in\cA(R)_\rf^\wedge$ by $g(s^i_j):=\beta^{-1} f(s^i_j)$. Note that $g$ is
  well-defined because all relations of $\cA(R)_\rf$ are
  homogeneous. Moreover, one has  $g(\det_q)=\beta^{-n} f(\det_q)=1$ and
  hence $g$ factors to a character of $\qfield_q[SL(n)]_\rf$.     
\end{proof}
 By Lemma \ref{charSLn} a character $f$ of the reflection equation
 algebra $\cA(R)_\rf$ gives rise to a character of the locally finite
 part $F_l(\Uq)$ if and only if $f(\det_q)\neq 0$ and the base field
 contains an $n$-th root of $f(\det_q)$. The following proposition
 provides an elementary yet somewhat surprising criterion for the first
 condition to hold. 
  \begin{prop}\label{neq0cond}
    For any character $f\in \cA(R)^\wedge_\rf$ the following are equivalent:
    \begin{enumerate}
       \item $f(\det_q)\neq 0$. \label{det1}
       \item $f:\cA(R)\rightarrow \qfield$ is convolution invertible.\label{det2}
       \item The matrix $M:=\big(f(s^i_j)\big)$ is invertible.\label{det3}
    \end{enumerate}
  \end{prop}
\begin{proof} 
Note first that (\ref{det2}) implies (\ref{det3}) by definition of the
coproduct of the bialgebra $\cA(R)$. We now show that (\ref{det1})
implies (\ref{det2}). Assume that $f(\det_q)\neq 0$ and let $\beta$ be
an $n$-th root of $f(\det_q)$ in a field extension $\qfield'$ of
$\qfield$. For now we consider all algebras to be defined over
$\qfield'$, yet the convolution inverse of $f$ we construct will be
defined over $\qfield$.
The character $g\in \cA(R)^\wedge_\rf$ defined by $g(s^i_j):=\beta^{-1}
f(s^i_j)$ factors to a character $g'$ of $\qfield'_q[SL(n)]_\rf$ because $g(\det_q)=1$. By
Lemma \ref{characterLem}.(2) the functional $g'$ is convolution
invertible, and we denote its convolution inverse by $\gbar'$. Define
$\gbar(a):=\gbar'(\Psi(a))$ where $\Psi:\cA(R)\rightarrow
\qfield'_q[SL(n)]$ denotes the canonical bialgebra homomorphism from
Subsection \ref{FRT-RE}. One checks that $g:\cA(R)\rightarrow \qfield'$
is convolution invertible with convolution inverse $\gbar$. For $m\in
\N_0$ let $\cA(R)_m\subseteq \cA(R)$ denote the homogeneous component of degree
$m$ of $\cA(R)$ defined over $\qfield$. Note that $\cA(R)_m$ is a
$\qfield$-subcoalgebra of $\cA(R)$ 
and hence $g(\cA(R)_m)\subseteq \beta^{-m}k$ implies
$\gbar(\cA(R)_m)\subseteq \beta^{m}k$. Indeed, any element $a\in
\cA(R)_m$ is contained in a subcoalgebra $A=span_k\{a_{ij}\,|\,1\le
  i,j\le N\}$ for some $a_{ij}\in  \cA(R)_m$ with
$\kow(a_{ij})=\sum_{h=1}^N a_{ih}\ot a_{hj}$. The implication
$\gbar(a)\in \beta^mk$ now follows from the fact that the matrix
$(\gbar(a_{ij}))$ is the inverse of the matrix $(g(a_{ij}))\in Mat_N(\beta^{-m}k)$.
Define now a functional $\fbar:\cA(R)\rightarrow
\qfield$ by $\fbar|_{\cA(R)_m}=\beta^{-m}\gbar|_{\cA(R)_m}$ and linearity. Note that
$\fbar$ is defined over $\qfield$ and not only over $\qfield'$. By
construction $\fbar$ is a convolution inverse of $f$.

It remains to show that (\ref{det3}) implies (\ref{det1}).  To this end
we first recall some well known properties of $\det_q$ and related
expressions in $\cA(R)$ (compare e.g.~\cite[9.2.1, 9.2.2]{b-KS}). 
\begin{enumerate}
  \item[(a)]  For any $\mu\in S_n$ the relation
    \begin{align*}
      \sum_{\sigma\in S_n} (-q)^{l(\sigma)}t^{\mu(1)}_{\sigma(1)} \dots
    t^{\mu(n)}_{\sigma(n)}=(-q)^{l(\mu)}{\det}_q 
    \end{align*}
    holds in $\cA(R)$.
  \item[(b)] For any elements $i_1,i_2,\dots,i_n\in
  \{1,2,\dots,n\}$ such that $i_k=i_l$ for some $k\neq l$ the relation
    \begin{align*}
      \sum_{\sigma\in S_n} (-q)^{l(\sigma)}t^{i_1}_{\sigma(1)} \dots
      t^{i_n}_{\sigma(n)} =0
  \end{align*}
  holds in $\cA(R)$.
\end{enumerate}
Assume that $M$ is invertible. We consider $M$ as an endomorphism of
$V$ defined by $M(v_i)=\sum_j v_j f(s^j_i)$. For $k=1,\dots,n$ define automorphisms
of $V^{\ot n}$ by
\begin{align*}
  M_k:=\id_{V^{\ot n-k}}\ot \big( R^{-1}_{V,V^{\ot k-1}}(M\ot \id_{V^{\ot k-1}}) R_{V,V^{\ot k-1}}\big)
\end{align*}
and set
\begin{align*}
  D:=M_n\dots M_2 M_1.
\end{align*}
Note that by construction and relations (\ref{fhom1}) one has
\begin{align*}
  D(v_{i_1}\ot \dots\ot v_{i_n})=\sum_{j_1,\dots,j_n}
 v_{j_1}\ot \dots\ot v_{j_n} f(t_{i_1}^{j_1}\dots t_{i_n}^{j_n}). 
\end{align*}
Applying this to the element 
  \begin{align*}
     Y:=\sum_{\sigma\in S_n} (-q)^{l(\sigma)}v_{\sigma(1)} \ot \dots \ot
     v_{\sigma(n)}\in V^{\ot n}
  \end{align*} 
and using properties (a) and (b) above one obtains
\begin{align*}
  D(Y)&=\sum_{j_1,\dots,j_n} v_{j_1}\ot \dots\ot v_{j_n} \,f\left(
      \sum_{\sigma\in S_n} (-q)^{l(\sigma)}t^{j_1}_{\sigma(1)} \dots t^{j_n}_{\sigma(n)}\right)\\
      &=f({\det}_q)\sum_{\mu\in S_n}(-q)^{l(\mu)}v_{\mu(1)}\ot \dots v_{\mu(n)}.
\end{align*}
As $D$ is injective one obtains $f(\det_q)\neq 0$.
\end{proof}
\begin{rema} 
The above proof is inspired by the arguments in
\cite[Section 3]{a-tD99}. Note, however, that in his paper T.~tom Dieck
restricts to a class of bottom-right triangular block-matrices.
\end{rema}
\begin{rema} It would be interesting to write the quantum
determinant explicitly in terms of the generators $\{s^i_j\}$ of the
reflection equation algebra $\cA(R)_\rf$. This  seems to be a
non-trivial combinatorial task.
\end{rema}
\begin{rema}
  Characters of $\cA(R)_\rf$ were classified explicitly in \cite{a-Mud02}. In
  view of Lemma \ref{charSLn} and Proposition \ref{neq0cond} this also
  gives the classification of characters of $F_l(U_q(\slfrak_n))$.
\end{rema}

\subsection{An example: quantum Grassmann manifolds $Gr(m,2m)$}\label{Grassmann}
Left coideal subalgebras $B$ of $\Uq$ corresponding to symmetric pairs of
Lie algebras are given explicitly in \cite[Section 7]{a-Letzter03}. As
an example we consider here the symmetric pair $\big(\slfrak_{2m},
\slfrak_{2m}\cap(\mathfrak{gl}_m\oplus \mathfrak{gl}_m)\big)$ labelled by
AIII(case 2) in the general classification \cite{a-Araki62}. The
corresponding symmetric space is the Grassmann manifold $Gr(m,2m)$ of
$m$-dimensional subspaces in $\C^{2m}$. We prefer here to work with the
corresponding right coideal subalgebra obtained via the antipode.

Assume that $n=2m$ is even. Fix a parameter $s\in \qfield$ and consider the
subalgebra $B^s$ of $\Uq$ generated by the following elements:
\begin{enumerate}
    \item[(i)] $\tau(\omega_i-\omega_{p(i)})$ \qquad for $i\neq m$,
    \item[(ii)] $B_i= y_i + t_i^{-1} x_{p(i)}$ for
    $i \neq  m$,
    \item[(iii)] $B_m= y_m  +t_m^{-1}x_{m}+ s\,t_m^{-1}$ 
\end{enumerate}
where $p(i):=n-i$. One verifies that $B^s$ is a right
coideal subalgebra of $\Uq$.  The pair $(\Uq,B^s)$ is a quantum
analogue of the pair $\big(U(\slfrak_n),
U(\slfrak_n\cap(\mathfrak{gl}_m\oplus \mathfrak{gl}_m))\big)$. 
Note that the family $\sigma(B^s)$, $s\in \qfield$, coincides with the family of
left coideal subalgebras for AIII(case 2) given in \cite[p.~284]{a-Letzter03}, up
to extension by elements in $\Uq^0$. Note moreover, that the subalgebra
of $\Uq^0$ generated by the elements in (i) coincides with
$span_k\{\tau(\lambda)\,|\,\lambda\in P, w_0\lambda=\lambda\}$ and
hence $t_it_{p(i)}^{-1}\in B^s$ for all $1\le i\le r$.

Recall that $r=\rank(\slfrak_n)=n-1$ and let $\omega_r:=-w_0\omega_1$
be the fundamental weight such that $V(\omega_r)$ is dual to the
vector representation $V(\omega_1)$. Recall the structure of centre of
the left coideal subalgebra $\sigma(B^s)$ described in Remark
\ref{QSP-remark}. By \cite[Proposition 9.1]{a-KL08} one has
$P^+_{Z(\sigma(B^s))}=P^+$ and hence 
\begin{align*}
  \dim(Z(B^s)\cap (\adl \Uq)\tau(-2\omega_r)\big)=1.
\end{align*}
Let $D\in Z(B^s)\cap (\adl \Uq)\tau(-2\omega_r)$ be a
nonzero element and set $C:=\sigma(D)$. By the results of Subsection \ref{solvingRE} there
exists a linear functional $f_C:C^V\rightarrow k$ such that
\begin{align}\label{f-def}
  C=f_C(c_{(1)})\sigma(\ltil_{\rfb_{21}}(c_{(2)}))
\end{align}
where $c=c_{\omega_1}$ is the quantum trace of $V(\omega_1)$. By
Corollary \ref{RE-character} one obtains a corresponding character
$g_{C,V}$ of the reflection equation algebra $\cA(R)_\rf$. In particular
the matrix $M:=(g_{C,V}(s^i_j))$ is a numerical solution of the reflection
equation \eqref{refleqn} with $R$ given by (\ref{R-explicit}). We will
prove the following lemma in the next subsection where we determine $M$ explicitly.  
\begin{lem}\label{inv-lem}
  The matrix $M:=(g_{C,V}(s^i_j))$ is invertible.
\end{lem}
We now summarise the results obtained about $B^s$, its centre, the
element $C=\sigma(D)$ and the corresponding character $g_{C,V}\in \cA(R)_\rf^\wedge$. 
In the following theorem we consider all algebras to be defined over a
suitable field extension $\qfield'$ of $\qfield$.
\begin{thm}  \label{centreMult}
  Let $\qfield'$ be a field extension of $\qfield$
  which contains an element  $\beta$ with $\beta^n=g_{C,V}(\det_q)$. Then
  the following hold:
  \begin{enumerate}
    \item $g(c_{i,j}):=\beta^{-1} g_{C,V}(s^i_j)$ defines a character of
    $\qfield'_q[SL(n)]_\rf\cong F_l(U_q(\slfrak_n))$.
    \item There exists a basis $\{D_\lambda\,|\,\lambda\in P^+\}$ of the
    centre $Z(B^s)$ such that the following hold for all
    $\mu,\lambda\in P^+$:
    \begin{enumerate}
      \item $D_\lambda\in Z(B^s)\cap (\adl \Uq)\tau(-2\lambda)$.
      \item $D_\lambda D_\mu=\sum_{\nu\in P^+}
      m^\nu_{\lambda,\mu}D_\nu$ \\where
      $m^\nu_{\lambda,\mu}:=\dim(\mathrm{Hom}_{\Uq}(V(\nu),V(\lambda)\ot V(\mu)))$. 
    \end{enumerate}
  \end{enumerate}
\end{thm}
\begin{proof}
 By Lemma \ref{inv-lem} the matrix $(g_{C,V}(s^i_j))$ is invertible and thus
 $g_{C,V}(\det_q)\neq 0$ by Proposition \ref{neq0cond}. Hence we may apply
 Lemma \ref{charSLn} to obtain the desired character of
 $\qfield'_q[SL(n)]_\rf\cong F_l(U_q(\slfrak_n))$ which proves (1). 

 To verify part (2) define $D_\lambda:=\sigma^{-1}(d_g(c_\lambda'))$
 for all $\lambda\in P^+$, where
 $c_\lambda'\in \qfield_q[SL(n)]^{inv}$ is the invariant element
 defined in (\ref{c'lambda-def}). Theorem \ref{RepThm} and
 \eqref{inTau2lambda} imply that
 \begin{align}\label{Dlambda}
   0\neq D_\lambda\in (\adl U)\tau(-2\lambda) \cap Z(\sigma^{-1}(B_{g_{C,V}})).
 \end{align}
 Note that up to a nonzero scalar factor the
 element $D_{\omega_r}$ defined in this way coincides with the element
 $D$. 
 The elements $\{D_\lambda\,|\,\lambda\in P^+\}$ satisfy the relations
 in (b) because the elements $\{c'_\lambda\,|\,\lambda\in P^+\}$ do,
 as explained in subsection \ref{NoumiCentre}. 

 It remains to show that the elements $\{D_\lambda\,|\,\lambda\in
 P^+\}$ form a basis of $Z(B^s)$. By Proposition \ref{Noumi-identify}
 one has $D_\lambda\in B^s$ for all $\lambda\in P^+$. Moreover, as
 $D\in Z(B^s)$ and (b) holds, all elements $D_\lambda$ for $\lambda\in
 P^+$ are invariant under the left adjoint action of $B^s$. By Lemma
 \ref{centreLem} and (\ref{Dlambda}) one hence obtains
 \begin{align*}
   D_\lambda\in Z(B^s)\cap (\adl \Uq)\tau(-2\lambda) \setminus \{0\}
 \end{align*}
 for all $\lambda\in P^+$. In view of (\ref{ZB1}) and (\ref{ZB2}) this
 implies that  $\{D_\lambda\,|\,\lambda\in P^+\}$ is a basis of $Z(B^s)$.
\end{proof}
\begin{rema}
  With some additional technical effort Lemma \ref{inv-lem} and hence
  Theorem \ref{centreMult} can also be shown 
  to hold for quantum symmetric pairs corresponding to arbitrary
  Grassmannian manifolds $Gr(m,n)$ where $2m\le n$. The corresponding
  quantum symmetric pair coideal subalgebras are defined in
  \cite[p.~284]{a-Letzter03} as AIII(case 1)/AIV. 
\end{rema}

\subsection{An explicit solution of the reflection equation}
In this subsection we will prove Lemma \ref{inv-lem} and explicitly
determine the numerical solution $M=(g_{C,V}(s^i_j))$ of the reflection
equation for the functional $f_C$ defined in 
(\ref{f-def}). For further calculations we first provide some explicit
formulae. Recall the basis $\{v_1,\dots,v_n\}$ with dual basis
$\{f_1,\dots,f_n\}$ chosen in Subsection \ref{characters}.
Note that
$\wght(v_i)=\omega_1-\alpha_1-\dots-\alpha_{i-1}=-\omega_{i-1}+\omega_i$
where we have set $\omega_0=\omega_n=0$. Hence one obtains
\begin{align}\label{2rhovi}
  (2\rho,\wght(v_i)) =n-2i+1.
\end{align}
It follows from (\ref{vec-rep}) that the matrix coefficients
$c_{i,j}:=c_{f_i,v_j}$ satisfy the relations
\begin{align*}
  c_{k,l}(x_iu)&= \delta_{i,k}c_{k+1,l}(u),&  c_{k,l}(ux_i)&=
  \delta_{i,l-1}c_{k,l-1}(u),\\
  c_{k,l}(y_iu)&= \delta_{i,k-1}c_{k-1,l}(u),&  c_{k,l}(uy_i)&=
  \delta_{i,l}c_{k,l+1}(u)
\end{align*}
for all $u\in \Uq$, and hence
\begin{align}
  (\ad_r^\ast x_i)c_{k,l}&=-\delta_{i,k}q^{-1}c_{k+1,l} +
  \delta_{i,l-1} q^{\delta_{k,l}-\delta_{k,l-1}} c_{k,l-1},\nonumber\\
  (\ad_r^\ast y_i)c_{k,l}&=-\delta_{i,k-1}
  q^{1+\delta_{k,l}-\delta_{k-1,l}}c_{k-1,l}+\delta_{i,l}c_{k,l+1},\label{adastckl}\\
  (\ad_r^\ast t_i)c_{k,l}&=q^{-\delta_{i+1,l}+\delta_{i,l}+\delta_{i+1,k}-\delta_{i,k}}c_{k,l}. \nonumber
\end{align} 
Define an $(n\times n)$-matrix $\Omega$ by  
\begin{align}\label{Om-def}
  D=\sum_{i,j=1}^n \Omega_{i,j} \,\ltil_{\rfb_{21}}(c_{i,j}).
\end{align}
Note that by definition (\ref{f-def}) of the functional $f_C$ and by (\ref{2rhovi}) we have 
\begin{align}\label{fOm}
  g_{C,V}(s^i_j)=f_C(c_{i,j})=q^{-(2\rho,\wght(v_i))} \Omega_{j,i}=q^{2i-n-1}\Omega_{j,i}.
\end{align}
Hence it suffices to determine $\Omega$ in order to determine $g_{C,V}$.
  Define an involutive automorphism of the weight lattice
  $\Theta:P\rightarrow P$, $\Theta(\mu)=w_0 \mu$. Note that
  $\Theta(\alpha_i)=-\alpha_{n-i}$.
  By \cite[Lemma 5]{a-Kolb08} the central element $D\in Z(B^s) \cap
  (\adl \Uq)\tau(-2\omega_r)$ satisfies  
  \begin{align}\label{domr-in}
    D\in \sum_{\makebox[0cm]{$\zeta,\xi\in Q^+\atop -\Theta(\zeta)+\xi\le \omega_1+\omega_r$}}(\adl
      U^+_\zeta)(\adl U^-_{-\xi}) \tau(-2 \omega_r).
  \end{align}
  Recall from Remark \ref{tau-as-ltil} that
  $\ltil_{\rfb_{21}}(c_{n,n})=\tau(-2\omega_r)$. By Proposition
  \ref{CalderoProp} the map $\ltil_{\rfb_{21}}$ is an isomorphism of left
  $\Uq$-modules and hence relations (\ref{domr-in}) and
  (\ref{adastckl}) imply that $\Omega_{i,j}=0$ if $n-i+1>j$.
 Hence we can write the matrix $\Omega$ in the form
\begin{align}\label{Oms-matrix}
  \Omega&=\left(\begin{array}{cc}
    0& F \\
    G& H
  \end{array}\right)
\end{align} 
where each entry $F,G,$ and $H$ is an $(m\times m)$-matrix. It follows
from the $\adl (t_i t_{p(i)}^{-1})$-invariance of $D$ for $1\leq i<m$
that both $F$ and $G$ are codiagonal and that $H$ is diagonal. To determine the remaining
entries of $F$, $G$, and $H$ one uses the following formulae which
immediately follow from (\ref{adastckl}).
\begin{lem}
  For $1\le i<m$ and $1\le j,k\le m$ the following relations hold:
  \begin{align}
    (\ad_r^\ast B_i)c_{k,n-k+1}&=-\delta_{i,k-1}q c_{k-1,n-k+1} + \delta_{i,k}q c_{k,n-k},\label{adast1}\\
    (\ad_r^\ast B_i)c_{n-k+1,k}&=\delta_{i,k} c_{n-k+1,k+1} -
    \delta_{i,k-1} c_{n-k+2,k},\label{adast2}\\
    (\ad_r^\ast B_{i})c_{n-j+1,n-j+1}&=-\delta_{i,j-1}q^{-1}
    c_{n-j+2,n-j+1} + \delta_{i,j} q c_{n-j+1,n-j}.\label{adast3} 
  \end{align}
\end{lem}
The $\adl(B_i)$-invariance
of $D$ for $1\le i<m$ now implies that all codiagonal entries of $F$
and $G$ are the same, respectively, and moreover that 
\begin{align}\label{OmRel1}
  \Omega_{i,i}=q^2 \Omega_{i+1,i+1},
\end{align}
if $m+1\le i\le n-1$. It remains to determine the relation between
$\Omega_{m+1,m+1}$ and the entries of $F$ or $G$. To this end one
calculates
\begin{align*}
  (\ad^\ast_r B_m)c_{m+1,m+1}&= -q^2 c_{m,m+1}+q^{-1}c_{m+1,m}+s\, c_{m+1,m+1},\\
  (\ad^\ast_r B_m)c_{m+1,m}&= -c_{m,m}+c_{m+1,m+1}+s\,q^{-2}c_{m+1,m} ,\\
  (\ad^\ast_r B_m)c_{m,m+1}&= -q^{-1}c_{m+1,m+1}+q^{-1}c_{m,m}+s\, q^2  c_{m,m+1}.
\end{align*}
Comparing coefficients in the equality $\adl(B_m)D=sD$ with respect to
the basis elements $\ltil_{\rfb_{21}}(c_{i,j})$ for $(i,j)=(m+1,m)$, $(m+1,m+1)$ one obtains 
\begin{align}\label{OmRel2}
  F=qG
\end{align}
 and 
\begin{align}\label{OmRel3}
 \Omega_{m+1,m+1}=s(q-q^{-1})\Omega_{m+1,m}.
\end{align}
The relations (\ref{OmRel1}), (\ref{OmRel2}), and (\ref{OmRel3})
determine $\Omega$ up to an overall scalar factor $\lambda\in
\qfield\setminus \{0\}$. One obtains
\begin{align*}
  \Omega_{i,j}=\lambda\begin{cases}
                  q& \mbox{if $i=n-j+1\le m$,}\\
                  1& \mbox{if $i=n-j+1\ge m$,}\\
                  s(q-q^{-1})q^{-2(i-m-1)}&\mbox{if $i=j\ge m+1$,}\\
                  0& \mbox{else.}
               \end{cases}
\end{align*}
 The character $g_{C,V}$ is obtained from (\ref{fOm}). We summarise the
 result of our calculation in the following proposition which also
 proves Lemma \ref{inv-lem}.
\begin{prop}
  Up to a nonzero scalar multiple the numerical solution $M=(g_{C,V}(s^i_j))$ of the
  reflection equation corresponding to the central element
  $C=\sigma(D)\in Z(\sigma(B^s))\cap (\adr \Uq)\tau(2\omega_r)$ is given by
  \begin{align*}
  g_{C,V}(s^i_j)=\begin{cases}
                  q^{2i-n}& \mbox{if $j=n-i+1\le m$,}\\
                  q^{2i-n-1}& \mbox{if $j=n-i+1\ge m$,}\\
                  s(q^2-1)&\mbox{if $i=j\ge m+1$,}\\
                  0& \mbox{else.}
               \end{cases}
  \end{align*}
\end{prop}

\begin{rema}
  Note the structural similarity between the matrix $M=(g_{C,V}(s^i_j))$ and
  the matrix defined in \cite[(2.14)]{a-NDS97},
  \cite[(6.13)]{a-DS99}. In principal, this similarity is not
  surprising. One argues along the lines of \cite[Section
  6]{a-Letzter99a} that up to a Hopf algebra automorphism of $\Uq$ and
  painstaking translation of conventions the coideal subalgebra
  constructed in \cite{a-NDS97} is a subalgebra of $B^s$. 
  Since the subspaces consisting of $\adl$-invariant elements in
  $F_l(\Uq)\simeq k_q[G]$ with respect to the two coideal subalgebras
  coincide (\cite[Theorem 2.6]{a-NDS97}, \cite[Theorem 6.6]{a-DS99} and \cite[Theorem
  7.7]{MSRI-Letzter}) the centres of the two coideal subalgebras also
  coincide. This in turn links the corresponding solutions of the reflection equation. 
\end{rema}
\begin{rema}  
The universal cylinder forms constructed in \cite{a-tD99} also have
the same structure as the character $g_{C,V}$ constructed for the
symmetric pair AIII(case 2) above. 
\end{rema}

\providecommand{\bysame}{\leavevmode\hbox to3em{\hrulefill}\thinspace}
\providecommand{\MR}{\relax\ifhmode\unskip\space\fi MR }
\providecommand{\MRhref}[2]{%
  \href{http://www.ams.org/mathscinet-getitem?mr=#1}{#2}
}
\providecommand{\href}[2]{#2}

\end{document}